\documentclass[11pt,
]{amsart}
\usepackage{tikz-cd}
\usepackage{
	amssymb, graphicx, color, cite 
}

\usepackage{mathtools}
\mathtoolsset{showonlyrefs}
\usepackage[linkcolor=blue, citecolor=black, breaklinks=true]{hyperref}

\usepackage[normalem]{ulem}

\renewcommand{\Im}{\textrm{Im}\,}

\renewcommand{\Re}{\textrm{Re}\,} 

\let\originalleft\left
\let\originalright\right
\renewcommand{\left}{\mathopen{}\mathclose\bgroup\originalleft}
\renewcommand{\right}{\aftergroup\egroup\originalright}

\date{}

\usepackage[T1]{fontenc}
\usepackage{lmodern}
\numberwithin{equation}{section}%

\newtheorem{theorem}{Theorem}[section]
\newtheorem{proposition}{Proposition}[section]
\newtheorem{lemma}{Lemma}[section]
\newtheorem{definition}{Definition}[section]
\newtheorem{corollary}{Corollary}[section]

\definecolor{HaimColor}{rgb}{0.92,0.19,0.6}
\definecolor{PlamenColor}{rgb}{0,0,1}

\theoremstyle{definition}
\newtheorem{remark}{Remark}[section]

\DeclareMathOperator{\supp}{supp}

\renewcommand{\sc}{semiclassical}

\DeclareMathOperator{\WFH}{WF_\textit{h}}

\DeclareMathOperator{\dist}{dist}
\newcommand{\eps}{\varepsilon}

\newcommand{\R}{{\mathbb R}}
\renewcommand{\S}{{\mathbb S}}

\newcommand{\Id}{\text{\rm Id}}
\renewcommand{\r}[1]{\eqref{#1}}

\newcommand{\be}[1]{\begin{equation}\label{#1}}
	\newcommand{\ee}{\end{equation}}
\renewcommand{\d}{\mathrm{d}}

\renewcommand{\i}{\mathrm{i}}

\newcommand{\p}{\partial}

\title[Viscoelastic wave equation with singular memory]{High-frequency wave propagation for the viscoelastic wave equation with singular memory}

\author[H.~Grebnev]{Haim Grebnev}
\address{Department of Mathematics, Purdue University, West Lafayette, IN 47907}

\author[P.~Stefanov]{Plamen Stefanov}
\address{Department of Mathematics, Purdue University, West Lafayette, IN 47907}
\thanks{Second author partly supported by NSF Grant DMS-2452757}

\begin{document}
\begin{abstract}
We study high-frequency propagation for a viscoelastic wave equation with spatially dependent hereditary memory written in relative-history form. The memory kernel is allowed to have an integrable singularity at the origin of the form 
$\mathfrak m(s,x)=s^{p-1}m(s,x)$, $0<p<1$, with the regular case $p=1$ included for comparison.   We use the half-amplitude propagation distance and the corresponding travel time as units of space and time. If $0<h\ll1$ denotes the wavelength parameter in these dimensionless variables, the transformed memory term carries the small factor $\varepsilon=h^{1-p}$. We construct exact solutions with a full two-scale geometric optics expansion in powers $h^{k+(1-p)\ell}$. The memory changes the propagation geometry through the instantaneous modulus $\sigma+\varepsilon \int_0^\infty \mathfrak{m}(s,\cdot)\, \d s$, while the singular part of the kernel contributes the complex coefficient $C_p=\Gamma(p)e^{\i\pi p/2}$ to the leading transport equation. For $0<p<1$, this produces fractional scales, frequency-dependent attenuation, and a   dispersive phase correction. When $p=1$, the fractional hierarchy disappears, the leading attenuation is frequency independent, and the corresponding transport phase correction vanishes.

We also derive a \textit{local} damped wave equation whose incoming high-frequency solutions approximate those of the hereditary equation with an $O(h)$ error in semiclassical $C^k$ norms. 
We prove unique recovery of $\sigma_{\mathfrak m}$ and the full temporal jet of $m(s,x)$ at $s=0$  from exterior observations of incoming waves for all incident directions and all $0<h\ll1$.  These quantities determine the complete high-frequency expansion modulo $O(h^\infty)$.

Finally, we establish well-posedness by a contraction semigroup and prove arbitrary finite-order Sobolev regularity for spatially dependent weakly singular kernels and a prescribed full prehistory, with explicit compatibility conditions and estimates uniform in $\varepsilon$. These estimates justify the geometric optics construction, in particular. 
\end{abstract}

\maketitle

\section{Introduction}
Consider the viscoelastic wave equation
\be{1}
\begin{dcases}
\partial_t^2 u-\nabla\cdot(\sigma(x)\nabla u)-\nabla\cdot\int_0^\infty
\mathfrak m(s,x)\bigl(\nabla u(t,x)-\nabla u(t-s,x)\bigr)\,\d s=0,
&t>0,\quad x\in\Omega,\\
u(0,x)=u_0(x),\quad u_t(0,x)=v_0(x),\quad u(t,x)|_{t<0}=u_-(t,x).
\end{dcases}
\ee
Here $\Omega\subset\R^n$, $n\ge2$, is a bounded smooth domain, $0<\sigma\in C^\infty(\bar\Omega)$ is the equilibrium stiffness, and $\mathfrak m\ge0$ is the memory kernel. We impose the homogeneous Dirichlet condition on $\p\Omega$. In the propagation part of the paper the $x$--support of $\mathfrak m$ is compactly contained in the region reached by the waves, and the waves do not meet the boundary, so the latter plays no role there. 
Unlike an ordinary wave equation, \r{1} is not determined by $(u_0,v_0)$ alone: the complete prehistory $u_-$ must also be prescribed. It is naturally encoded by the history variable
\be{eta}
\eta^t(s)=u(t)-u(t-s),\qquad t\ge0,\quad s\ge0,
\ee
whose spatial gradient is the relative strain history of age $s$ at time $t$. Apart from the regularity and matching conditions needed at $t=0$, the prehistory is not required to solve an equation for negative time. The precise compatibility conditions between $u_-$, the initial data, and a possible source are part of the higher regularity theory developed in Section~\ref{sec_wp}.

In the scalar setting of \r{1}, denoting the stress flux by $\tau(t,x)$, the Boltzmann constitutive law takes the compact form
\[
\tau(t,x)=\sigma(x)\nabla u(t,x)+\int_0^\infty\mathfrak m(s,x)\nabla\eta^t(s,x)\,\d s,
\]
and \r{1} is the balance law $\partial_t^2u-\nabla\cdot\tau=0$. For the physical foundations of this constitutive law, see \cite{ColemanNoll1961,Lakes2009}. The relative-history form is used explicitly in \cite{HrusaR_85,Chep_Pata}, while equivalent strain and strain-rate convolution formulations, as well as more general hereditary models, are discussed in \cite{HrusaNR_88,Hanyga2002PropagationPulses}.

We are interested in kernels that are smooth for $s>0$ but have an integrable conormal singularity at the origin. More precisely, we assume
\begin{equation}\label{m0}
\mathfrak m(s,x)=s^{p-1}m(s,x),\qquad 0\le m\in C^\infty([0,\infty)\times\R^n),\qquad 0<p<1,
\end{equation}
and we also discuss the regular case $p=1$. In addition, we assume that the kernel is decreasing in $s$ and is integrable at infinity, together with the spatial derivatives required in the corresponding results, see section~\ref{sec_wp}. Singular memory kernels and their influence on wave propagation have been studied in \cite{Hanyga2001SingularMemory,Hanyga2002PropagationPulses,HanygaSeredynska1999MemoryKernelSingularityI,HanygaSeredynska2002SingularMemory,HrusaR_85}. They arise naturally in fractional constitutive models \cite{BagleyTorvik1983} and produce power-law attenuation, dispersion, phase shifts, and smoothing of wave fronts; see also \cite{Szabo1994PowerLaw,HolmNasholm2014,BrouckeOparnica2022} and the discussion in Section~\ref{sec_phys}.

The present value of the strain in the memory term contributes the nonnegative quantity
\be{sigma_m}
\sigma_{\mathfrak m}(x)=\int_0^\infty\mathfrak m(s,x)\,\d s
\ee
to the stiffness. The associated relaxation modulus is $a(s,x)=\sigma(x)+\int_s^\infty\mathfrak m(r,x)\,\d r$, so that $-\partial_s a=\mathfrak m$. Thus $a(\infty,x)=\sigma(x)$ is the equilibrium modulus, while $a(0,x)=\sigma(x)+\sigma_{\mathfrak m}(x)$ is the instantaneous modulus and determines the corresponding propagation speed; see \cite[Chapter~2]{Lakes2009} for this standard terminology. In particular, the  assumption $\sigma=1$ we adopt later fixes only the equilibrium elastic background; the relaxation behavior may remain spatially heterogeneous through $\mathfrak m$. We keep the present-time and convolution contributions separate in order to track how the relaxing part transforms under the subsequent non-dimensionalization. The remaining convolution has multiplier of order $-p$ in the temporal frequency. Combined with the two spatial derivatives, it is therefore of order $2-p$ and produces an attenuation rate increasing as the frequency to the power $1-p$.

Closest to the present work in its combination of a viscoelastic wave model, weakly singular kernels, forward analysis, inverse estimation, and numerical reconstruction is \cite{KaltenbacherEtAl2022}. That work treats general linear elastic systems with temporal kernels, proves well-posedness of the state and adjoint equations, and develops a PDE-constrained optimization method for estimating the kernels from additional observations. Here the kernel depends also on $x$, and our objective is different: we analyze its effect on high-frequency propagation and prove uniqueness in recovering its instantaneous modulus and its singular temporal structure from wave measurements.

To describe the high-frequency regime, we use the half-amplitude propagation distance and the corresponding travel time as units of space and time, and let $h>0$ denote the dimensionless wavelength parameter. We write the corresponding dimensionless frequency as $\lambda/h$, where
\be{lambda}
0<\lambda_{\min}\le\lambda\le\lambda_{\max}<\infty.
\ee
The constants in the construction are uniform for $\lambda$ in this fixed band. We send into the medium a wave of the form $\chi(x)e^{\i\lambda(-t+x\cdot\omega)/h}$ whose support enters the $x$--support of $\mathfrak m$. At high frequency the physical half-amplitude distance shrinks but still contains many wavelengths, so it is a natural unit for the propagation problem. Using this distance and the corresponding travel time as the units of space and time preserves the normalized wave speed and makes the relevant propagation interval of order one, see section~\ref{sec_rescale}. After transforming the conormal kernel and renaming the rescaled variables and its smooth profile, the dimensionless memory term carries the factor
\be{critical-scaling}
\varepsilon=h^{1-p}.
\ee
Here and below $\mathfrak m$ denotes the transformed dimensionless kernel profile. This is a change of units for a fixed material, rather than an $h$--dependent change of its constitutive law; see Section~\ref{sec_rescale} again. For the analysis, we first regard $h$ and $\varepsilon$ as independent parameters and impose \r{critical-scaling} afterward. The instantaneous modulus in these variables is $\sigma+\varepsilon\sigma_{\mathfrak m}$, while the attenuation accumulated along a ray is of order one.

This regime presents several related difficulties. The singular convolution generates fractional powers of $h$, and the parameters $h$ and $\varepsilon=h^{1-p}$ have to be tracked simultaneously. Moreover, the memory changes not only the transport equations but also the eikonal equation: the phase itself depends on $\varepsilon$. Consequently, one must construct and control an expansion involving the exponents $h^{k+(1-p)\ell}$ and then correct it to an exact solution using estimates uniform in $\varepsilon$, hence uniform as $h\to0$ after setting $\varepsilon=h^{1-p}$. In the forward theory, the full prehistory and the spatial dependence of the kernel must also be incorporated into the energy and generator domains, and higher time regularity requires compatibility conditions that contain the prehistory contribution.

Our \textbf{first main result}, Theorem~\ref{thm_exp}, gives an exact solution with a full geometric optics expansion in integral powers of $h$ whose phase and amplitudes depend smoothly on $\varepsilon$. Expanding this dependence produces a two-scale asymptotic hierarchy in powers $h^{k+(1-p)\ell}$. The $\varepsilon$--dependent phase records the modification of the propagation geometry by the instantaneous modulus, while the first transport equation exhibits the leading attenuation and phase shift produced by $m(0,x)$. The expansion also quantifies the frequency dispersion caused by the singularity. These conclusions are not merely extensions of the regular case. When $p=1$, the second scale and the fractional hierarchy disappear, and the leading memory contribution produces frequency-independent attenuation with no corresponding transport phase shift. For $0<p<1$, the singularity produces the additional scale $h^{1-p}$, fractional powers, frequency-dependent attenuation, and a nonzero phase correction responsible for dispersion. For clarity, the principal construction is carried out with $\sigma=1$; the extension to a variable background stiffness is discussed in Section~\ref{sec_var}.

A consequence of the expansion is a \textit{local} differential approximation of the hereditary equation at a fixed $\lambda$. Theorem~\ref{thm_appr} constructs a damped wave equation involving $\sigma_{\mathfrak m}$ and $m(0,x)$ whose incoming geometric optics solution differs from that of \r{1} by $O(h)$ in semiclassical $C^k$ norms. Besides identifying the first effective local model, this approximation avoids storing and repeatedly integrating the complete history in numerical propagation. It may also be useful for time reversal, since the local damped equation can be evolved backward over a finite interval whereas the hereditary equation would require knowledge of a future history.

Our \textbf{second main result} concerns the inverse problem. We observe the waves after they leave the support of the memory, for all incident directions and all sufficiently small $h$. Theorem~\ref{thm_IP} shows that these measurements determine $\sigma_{\mathfrak m}(x)$ and the full jet $\partial_s^j m(0,x)$, $j\ge0$. The phase first determines the X-ray transform of $\sigma_{\mathfrak m}$; after the phase is known, the successive transport equations determine the Taylor coefficients of $m$ through 
the X-ray transform. The attenuation arising in this recursion may be complex; injectivity in the Euclidean setting follows from \cite{Novikov}. A perturbation of $m$ that is flat at $s=0$ changes the high-frequency expansion only by $O(h^\infty)$, which explains both the scope and the natural limitation of the recovery result.

Earlier, Romanov \cite{Romanov2014Coefficients}  obtained inverse results for viscoelastic systems with smooth convolution kernels ($p=1$ in our notation) from boundary impulse measurements, recovering, under suitable geometric assumptions, the coefficients and the temporal jet of the memory kernel at the origin. More recently, geometric optics and inverse recovery for a spatially dependent but smooth memory kernel have been studied in \cite{Covi_26}.   In that setting, the singularity in \r{m0}, the critical scaling \r{critical-scaling},  the resulting fractional two-scale structure, and the resulting qualitative effects are absent. In \cite{acosta2025}, fractional attenuation is modeled by a lower-order term $a(x)\p_t^\alpha u$, $\alpha\in(0,1)$,  whereas the memory operator in the present paper behaves as $\p_t^{-p}\Delta$ and has order $2-p\in(1,2)$. For hyperbolic equations with a spatially lower-order memory term, reconstruction of the memory kernel from partial boundary measurements was established in \cite{BukhgeimDyatlovUhlmann2002}, and a related unique continuation result was proved in \cite{BukhgeimDyatlovUhlmann2007}. In contrast with the optimization approach of \cite{KaltenbacherEtAl2022}, our inverse result is an exact uniqueness theorem derived from the high-frequency phase and amplitudes.

Our \textbf{third main result} is the well-posedness and higher regularity theory for \r{1}. In particular, it is needed to justify the construction but we prove much more than that. We represent the weak evolution by a contraction semigroup on an energy space containing the history variable. This semigroup formulation follows the history-space approach used by Chepyzhov and Pata \cite{Chep_Pata}, building on Dafermos \cite{Dafermos_70,Dafermos_76}; their memory kernel, however, is regular and independent of $x$. We then characterize the regular data and compatibility conditions that yield $C^k$--Sobolev solutions for arbitrary finite $k$, including the contributions of the prescribed prehistory and a source, and obtain the estimates required uniformly for $0\le\varepsilon\le\varepsilon_0$. Classical solutions for a one-dimensional regular finite-delay wave equation were obtained by Hille--Yosida methods in \cite{Sinestrari1999}, while \cite{KaltenbacherEtAl2022} proves energy well-posedness for state and adjoint problems with temporal weakly singular kernels. To our knowledge, the combination used here --- a spatially dependent weakly singular kernel, a prescribed full prehistory, explicit arbitrary-order Sobolev compatibility conditions, and estimates uniform in the scaling parameter --- has not previously been established.

The paper is organized as follows. Section~\ref{sec_stat} gives a stationary analysis motivating the scaling and the propagation effects. Section~\ref{sec_GO} constructs and analyzes the geometric optics expansion and proves the local differential approximation, after which the principal effects are illustrated numerically. Section~\ref{sec_IP} treats the inverse problem. Further physical and mathematical remarks follow, and Section~\ref{sec_wp} develops the semigroup and higher regularity theory and completes the justification of the asymptotic construction.

\section{A stationary approach} \label{sec_stat}
We start with a stationary approach. Some of the analysis here is on a heuristic level. Assume $\sigma=1$.  
Look for solutions of the form $u(t,x)=e^{-\i t\lambda}v(x,\lambda)$. Then $v$ would solve the following stationary equation, obtained by formally taking the Fourier transform with respect to $t$:
\begin{equation}   \label{3.1}
(-\Delta-\lambda^2  +  \nabla\cdot ( \mu(\lambda,x)-\mu(0,x) )\nabla )v=0,
\end{equation}
where
\begin{equation}   \label{mu}
\mu(\lambda,x) = \int_0^\infty e^{\i\lambda s}\mathfrak{m}(s,x)\,\d s
\end{equation}
is just the rescaled inverse Fourier transform with respect to $s$ of $\mathfrak{m}$, extended as $0$ for $s<0$.
Note that $\mu$ is complex-valued, and $\Im\mu\ge0$ for $\lambda\gg0$. We have $\mu(0,x) = \sigma_{\mathfrak{m}}(x)$. 
 
\begin{lemma}\label{lemma_m_exp}
	We have the following asymptotic expansion 
\begin{equation}   \label{alpha_exp}
	\mu(\lambda,x)\sim   \lambda^{-p}\sum_{k\ge0}   C_{p+k}\frac{m_k(x)}{k!}\lambda^{-k}, \quad \text{as $\lambda\gg1$},
\end{equation}
where $m_k(x)=  \partial_s^km|_{s=0} $, 
and $C_{p+k}:= e^{\i \pi (p+k)/2}\Gamma(p+k) = i^{k}p(p+1)\dots (p+k-1)C_p $. In particular, 
\[
\mu(\lambda,x)\sim   {C_p}\, \lambda^{-p} m_0(x)\big(1+O(\lambda^{-1})\big) , \quad C_p := e^{\i\pi p/2} \Gamma(p).
\] 
\end{lemma}

\begin{proof}
We have, see \cite[p.~110]{Friedlander1998}, \cite[p.~72]{Hormander1},
	\[
	\big(s_+^{p-1}\big)\hat{\big.} =   \Gamma(p) \left( e^{-\i \pi p/2} \lambda_+^{-p} + e^{\i \pi p/2} \lambda_-^{-p} \right), 
	\]
where $\lambda_-=(-\lambda)_+$. In particular, 
	\begin{equation}   \label{Cp0}
	\big(s_+^{p-1}\big)\hat{\big.} = \overline{C_p} \lambda^{-p} \quad  \text{for $\lambda>0$}. 
	\end{equation}
Take the Taylor expansion $m(s,x) \sim \sum_{k=0}^\infty \frac{s^k}{k!} m_k (x)$ with respect to $s$. Multiply by $s_+^{p-1}$ and take the Fourier transform to get \r{alpha_exp}, 
where we used  that $\mu(\lambda)= \hat {\mathfrak{m}}(-\lambda)$. This completes the proof. 
\end{proof}

\begin{remark}
Expansion \r{alpha_exp} establishes a one-to-one correspondence between the Taylor expansion of $m(s,x)$ at $s=0$ and polyhomogeneous symbols in $\R$ of order $-p$, restricted to $\lambda>0$. Then we can view $\mathfrak{m}$ as a conormal distribution to $s=0$ with such a symbol, supported in $s\ge0$, and \r{m0} shows that $\mathfrak{m}$ is the general form of such distributions. 
\end{remark}

Notice that 
\be{ReIm_mu}
	\Re \mu \sim m_0(x)  \Gamma(p)\cos(p \pi /2)\lambda^{-p}, \quad \Im\mu\sim m_0(x) \Gamma(p)\sin(p\pi/2) \lambda^{-p},
\ee
where $m_0(x)=m(0,x)$. 
The real part would create a phase shift slightly slowing the wave down. In particular, when $p=1$, the leading real part vanishes, so there is no leading-order phase shift. 
The imaginary part creates attenuation. 
	
\subsection{Memory $\mathfrak m$ independent of $x$}
Assume below that $\mathfrak{m}(s)$ is independent of $x$. Then \r{3.1} implies 
\begin{equation}   \label{3.2}
\left(-\Delta - \frac{\lambda^2}{1+\mu(0)-  \mu(\lambda)} \right)v=0.
\end{equation}
Here $\mu \sim \lambda^{-p}$ as $\lambda \to\infty$ by \r{ReIm_mu}. 
 Then we have the effective frequency 
\be{l_eff}
\lambda_{\rm eff}\coloneq \lambda (1+\mu(0)-\mu(\lambda))^{-1/2} = \frac{\lambda}{c}  \Big(1 +\frac{1}{2c^2} \mu(\lambda)+\dots \Big), \quad c:= \sqrt{1+\mu(0) }. 
\ee
Look for a solution $v=e^{\i \lambda_{\rm eff} x\cdot\omega}$. 
The solutions then look like this:
\begin{equation}   \label{3.4}
v\sim e^{\i\lambda x\cdot\omega/c +\i\frac1{2c^3}  \lambda\mu(\lambda)x\cdot\omega+O(\lambda|\mu|^2/c^4)}.
\end{equation}
When $p=1$, $\mathfrak{m}$ is not singular and $\mu =\i m_0 /\lambda +O(\lambda^{-2})$, so we get 
\[
v\sim e^{\i \lambda x\cdot\omega/c -\frac1{2c^3}m_0 x\cdot\omega},   \quad \text{therefore} \; 
u\sim e^{\i\lambda (-t+x\cdot\omega/c)-\frac1{2c^3}m_0 x\cdot\omega} .  
\]
The effect of the memory term now is to create the attenuation term $e^{-\frac1{2c^3} m_0 x\cdot\omega}$ along the way. Note that there is no $\lambda$ in that term, i.e., the attenuation rate is independent of  the wavelength. 

When $0<p<1$, we have $\mu = C_p\lambda^{-p}m_0+O(\lambda^{-p-1})$, and \r{3.4} implies 
\be{v}
v\sim \exp\left(\i \lambda x\cdot\omega/c +\i \frac{1}{2c^3} C_p\lambda^{1-p}m_0 x\cdot\omega+O(\lambda^{1-2p})\right),
\ee  
therefore,
\be{u}
u\sim \exp\left(\i \lambda(-t+ x\cdot\omega/c) +\i\frac1{2c^3} C_p\lambda^{1-p}m_0  x\cdot\omega + O(\lambda^{1-2p})\right) .
\ee
Notice that $O(\lambda^{1-2p})$ is not small when $0<p\le 1/2$; thus more terms are needed for a small remainder. 

The real part of $C_p$ produces the wave-number correction $\delta k
=\frac{1}{2c^3}\Re C_p \lambda^{1-p}m_0$. 
Equivalently, after propagating a distance $L$, the accumulated phase
correction is $\delta k\,L$. Relative to the leading wave number
$k_0=\lambda/c$, this correction has size
${\delta k}/{k_0}=\frac{1}{2c^2} \Re C_p m_0\lambda^{-p}$. 
The imaginary part of $C_p$ gives the attenuation rate 
$\frac1{2c^3}\Im C_p\,\lambda^{1-p}m_0$. Note that these two rates depend on $\lambda$ now. The distance $d$ over which the attenuation is $1/2$ is given by 
\[
d= \frac{2c^3\log 2}{m_0 \,\Im C_p} \lambda^{p-1}.
\]
Writing $\ell=2\pi/\lambda$ for the wavelength in the original units, the half-amplitude distance is proportional to $\ell^{1-p}/m_0$. Thus $d\to0$ as $\ell\to0$, although $d/\ell\to\infty$: a high-frequency wave propagates through many wavelengths but is attenuated over a macroscopically vanishing distance. We therefore use $d$ as the unit of space and the corresponding travel time as the unit of time. In the resulting variables the dimensionless wavelength is $h=\ell/d$, and the transformed conormal kernel carries the factor $h^{1-p}$. Notice that $c$ also depends on $m$, but when $m$ is small, $c\sim1$, which does not change the arguments above. 

\subsection{A rescaled version} \label{sec_2.2}  
After this nondimensionalization and removal of the tildes, the frequency is written as $\lambda/h$, with the new $\lambda$ varying in \r{lambda}, and the memory term appears with the factor $h^{1-p}$. Even when $\mathfrak{m}$ depends on $x$, for the corresponding $\mu_h(\lambda,x) := h^{1-p} \mu(\lambda/h,x)$, we have
\begin{equation}   \label{alpha_exp1}
\mu_h(\lambda,x) \sim \lambda^{-p} h\sum_{k\ge0} C_{p+k}\frac{m_k(x)}{k!}\lambda^{-k}h^{k}, \quad \text{as $0<h\ll1$}.
\end{equation}
Thus  $\mu_h$ is a polyhomogeneous \sc\ symbol of order one with a small parameter $h$, uniformly for $\lambda$ in the fixed
band \r{lambda}.   It equals $\lambda^{1-p}$ times a function of $h/\lambda$. 


We have the effective speed
\begin{equation}   \label{c(h)}
c(h) = \sqrt{1+h^{1-p}\mu(0)}.
\end{equation}
Note first that $c(h) =1+\frac12 h^{1-p} \mu(0)+O(h^{2-2p})$. Next, the effective semiclassical frequency is 
\be{l_eff1}
\lambda_{\rm eff}  = \frac{\lambda}{c(h)}  \Big(1 +\frac{1}{2c^2(h)} \mu_h(\lambda)+ \frac{3}{8c^4(h)} \mu^2_h(\lambda)+ \dots \Big),  
\ee
see also \r{alpha_exp1}. 
 
Then the phase function in \r{u} takes the form
\begin{equation}\label{phi_exp}
\Phi = \frac{\lambda}{h}(-t+ x\cdot\omega/c(h)) + \frac{\lambda^{1-p}}{2c^3(h)} C_p  m_0 \, x\cdot\omega+ O(h).
\end{equation}
The  factor $c(h)$ above indicates an increase of the phase velocity by $\frac12 h^{1-p}\mu(0)+O(h^{2-2p}) $.  
The second term indicates an additional (smaller) speed correction by $-\frac{h}{2c(h)} (\Re C_p) \lambda^{-p} m_0$ to leading order, and an attenuation at the rate of $\frac1{2c^3(h)} (\Im C_p) \lambda^{1-p} m_0$. 
The temporal frequency remains unchanged, $\lambda/(2\pi h)$.

Another observation, crucial in the time-dependent case, is that it takes many (approximately $1/(1-p)$) terms in the asymptotic expansion of $c(h)$ to get to an error $O(h)$. Before we get to $o(h)$, the contribution of those ``large'' terms in the expansion of $c(h)$ to the first term in \r{phi_exp} creates fast oscillations in $u\sim e^{\i \phi}$, see \r{u}, and cannot be considered part of the amplitude. More precisely, $e^{\i ah^q}\sim 1+\i a h^q+O(h^{2q})$ with $a$ real only when $q>0$. This motivates our approach of considering a phase function that depends on $h$ as well. We also see that the expansion of the phase is in fractional powers of $h$ while the expansion of the amplitude is in integer powers if the phase is considered dependent on the 
parameter $\eps=h^{1-p}$. 

We present another, \sc\ point of view of the scaling. Returning to the case in which $\mathfrak{m}$ also depends on $x$, equation \r{3.1} takes the form
\begin{equation}   \label{3.1a}
(-h^2\Delta-\lambda^2  +   h\nabla\cdot h^{1-p}(\mu(\lambda/h,x)- \mu(0,x))h\nabla )u=f(x,\lambda,h),
\end{equation}
where we put a source on the right. Since
$h^{1-p}\mu(\lambda/h,x)\sim
hC_p\lambda^{-p}m_0(x)$, equation~\r{3.1a} reduces to
\[
\big(-h^2\Delta-\lambda^2 - h^{1-p} h\nabla\cdot \mu(0,x)h\nabla  + C_p h \lambda^{-p}h\nabla\cdot m_0(x)  h\nabla \big)u=f(x,\lambda,h),
\]
ignoring the lower-order terms.

This shows that, semiclassically, the new scaling reduces the memory term to a real-valued term that is lower by the fractional degree $1-p$, plus a complex-valued term of degree one lower. The latter affects the leading amplitude, whereas the former is stronger and also affects the geometry.

\section{Geometric optics}\label{sec_GO}

Assume that $\mathfrak{m}$ satisfies \r{m0}, and the support condition 
\begin{equation}   \label{supp}
\supp_x\mathfrak{m}\subset B(0,R)
\end{equation}
with some $R>0$.  

We assume \eqref{m-int} throughout this section. The monotonicity
condition \eqref{m-monotone} and the higher-order conditions
\eqref{m-int-k} are not needed for the formal construction; they enter
in Theorem~\ref{thm_exp} through the well-posedness and error estimates.
We assume $\sigma=1$, see also section~\ref{sec_var}. 
We carry out the construction in $\R^n$. The same construction applies in a bounded domain $\Omega$ on any time interval during which the waves under consideration do not meet $\partial\Omega$.

From now on, space and time are the dimensionless variables obtained by using the half-amplitude propagation distance and the corresponding travel time as their respective units, see the Introduction and section~\ref{sec_rescale}. We again write $\mathfrak m$ for the transformed kernel profile. Introduce
\be{varepsilon}
\varepsilon:=h^{1-p} 
\ee
for the dimensionless coefficient of the memory term. Accordingly, $\varepsilon\mathfrak m(s,x)=h^{1-p}s^{p-1}m(s,x)=(h/s)^{1-p}m(s,x)$; see \r{m0}. We also set
\begin{equation}   \label{sigma_eps}
\sigma_\eps := 1+\eps\sigma_{\mathfrak{m}}(x) =1+h^{1-p}\sigma_{\mathfrak{m}}(x),
\end{equation}
Here $\sigma_{\mathfrak m}$ is defined in \r{sigma_m}, and $\sigma_\eps$ should not be confused with $\sigma_{\mathfrak m}$. In the geometric optics construction below, we view $h$ and $\eps$ as two independent small parameters. Then we modify \r{1} to
\be{1b}
\partial_t^2 u	-\Delta u - \eps\nabla \cdot \int_0^\infty \mathfrak{m}(s,x)\big( \nabla u(t,x) - \nabla u(t-s,x)\big)\,\d s=0, \quad \text{in $\R_+\times\R^n$}. 
\ee
Equation \r{1b}  can be recast as
\be{1c}
\partial_t^2 u	- \Delta_{\sigma_\eps }  u + h^{1-p} \nabla\cdot \int_0^\infty  \mathfrak{m}(s,x)\nabla u(t-s,x) \,\d s=0, \quad \text{in $\R_+\times\R^n$},
\ee
where 
\[
 \Delta_{\sigma_\eps }:= \nabla\cdot(1+h^{1-p}\sigma_{\mathfrak{m}}(x) )\nabla.
\] 
 We introduce $P_h$ and $M$ by writing \r{1c} as $P_hu:= (\partial_t^2-\Delta_{\sigma_\eps }) u +h^{1-p}M u=0 $.

 Choose $\chi\in C_0^\infty(\R^n)$ with $\supp\chi\cap \overline{B(0,R)}=\emptyset$. Set 
\begin{equation}   \label{t0}
t_0 \coloneq \dist (\supp\chi, B(0,R))>0. 
\end{equation}
We choose the initial conditions  
\begin{equation}   \label{IC1}
u|_{t=0}=2\chi(x) e^{\i\lambda  x\cdot\omega /h}, \quad u_t|_{t=0}= 0 , \quad u_-(t)|_{|x|<R}=0, \; t<0.
\end{equation}
This corresponds to $\eta^0(s)=0$ in \r{meq2d}, where $\eta^t$ is an equivalence class. For a short time $t\in[0,t_0]$ before the $x$-support of $u$ enters $B(0,R)$, the solution can be expressed by the standard geometric optics expansion and splits into two parts: $u_\text{out}$, propagating away from $B(0,R)$, and $u_\text{in}$, propagating toward it; see Figure~\ref{fig_incoming}. We are interested in the latter. 

 Fix $T>0$. By finite speed of propagation, both for the transport construction and for the exact solution, we may choose $R_T>R$ so that all causal supports under consideration for $0\le t\le T$ are contained in $B(0,R_T)$; see also section~\ref{sec_wp}.

\begin{figure}[tbp]
\includegraphics{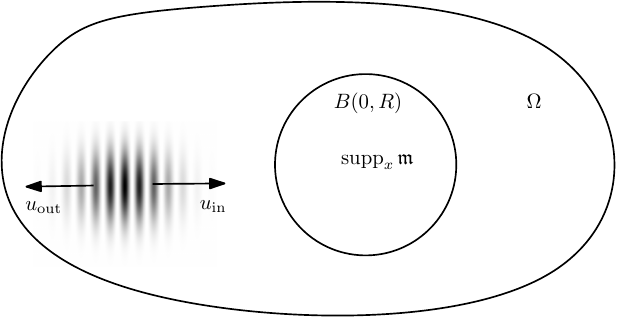}
\caption{The initial wave packet $u_\textrm{in}$ probing $\mathfrak{m}$.}\label{fig_incoming} 
\end{figure}

The ansatz for $u$ now is
\be{ansatz}
u_\text{in}\sim e^{\i\lambda \phi(t,x,\eps)/h} a(t,x,h,\eps), \quad a\sim a_0(t,x,\eps) +ha_1(t,x,\eps)+\cdots,
\ee
as $0<h\ll1$, where 
we allow $\phi$ and $a_j$ to depend on $\eps$, while $a$ also depends on $\lambda$. Their dependence on $\eps$ is smooth up to $\eps=0$; in fact, they have Taylor expansions at $\eps=0$. The dependence on $\lambda$ satisfying \r{lambda} is uniform. 

\subsection{Construction of the phase} \label{sec_phase}

We take $\phi$ to be the solution of the eikonal equation 
\begin{equation}   \label{eik1}
(\partial_t \phi)^2 - \sigma_\eps(x) |\nabla\phi|^2=0, \quad x\in B(0,R_T), \qquad \phi|_{x\cdot\omega<-R}= -t+x\cdot\omega.
\end{equation}
We choose 
\begin{equation}   \label{phase}
\phi = -t+ \psi(x,\eps).
\end{equation}
Then $\psi$ would satisfy the eikonal equation
\begin{equation}   \label{eik2}
\sigma_\eps|\nabla\psi|^2=1, \quad x\in B(0,R_T), \qquad \psi|_{x\cdot\omega<-R}= x\cdot\omega.
\end{equation}

Standard Hamilton--Jacobi theory provides a smooth solution of \r{eik2} in $B(0,R_T)$ when $0<\eps\ll1$.

\begin{proposition} 
\label{prop:phase_smooth}

There exists $\varepsilon_0>0$ such that, for every $0\le \varepsilon\le \varepsilon_0$,
the eikonal equation \r{eik2} 
admits a solution $\psi\in C^\infty(B(0,R_T))$, which also depends smoothly on $\varepsilon \in [0,\varepsilon_0]$ and on $\omega$. Consequently, $\psi$ admits a Taylor expansion in $\varepsilon$,
\[
\psi(x,\eps)
\sim \sum_{j\ge0}\varepsilon^j\psi_j(x)
\qquad \text{in } C^\infty(B(0,R_T)),
\]
with
\[
\psi_0(x)=x\cdot\omega,
\qquad
\psi_j(x)=\frac1{j!}\,\partial_\varepsilon^j\psi(x,\eps)\big|_{\varepsilon=0},
\]
and the coefficients satisfy
\begin{equation}\label{phi_rec_short}
2\,\omega\cdot\nabla \psi_j
=
(-1)^j\sigma_{\mathfrak{m}}^j
-
\sum_{\ell=1}^{j-1}\nabla\psi_\ell\cdot\nabla\psi_{j-\ell},
\qquad j\ge1,
\end{equation}
together with
\[
\psi_j(x)=0
\qquad \text{for } x\in B(0,R_T),\quad x\cdot\omega<-R.
\]
\end{proposition}

\begin{proof} Fix $\omega_0\in \S^{n-1}$ and set
\[
\Sigma_-:=\{x\in\mathbb R^n;\ x\cdot\omega_0=-R\}.
\]
On $\Sigma_-$, for $\omega\in\S^{n-1}$ close enough to $\omega_0$, prescribe the incoming data
\[
\psi|_{\Sigma_-}=x\cdot\omega|_{\Sigma_-}.
\]
Since $\sigma_{\mathfrak{m}}=0$ near $\Sigma_-$, the hypersurface $\Sigma_-$ is noncharacteristic for
\eqref{eik2}, and the corresponding characteristic system is a smooth perturbation,
for $\varepsilon$ small, of the free system associated with $\psi_0(x)=x\cdot\omega$.
By the method of characteristics, there exists $\varepsilon_0>0$ such that for
$0\le \varepsilon\le \varepsilon_0$ the problem
\eqref{eik2} has a unique smooth solution in a neighborhood
of the incoming rays through $\Sigma_-$, with $\nabla \psi |_{\Sigma_-} =\omega$.  
 For $\varepsilon_0$ small enough, this neighborhood, together with the free incoming region $x\cdot\omega<-R$, contains $B(0,R_T)$. This proves existence on $B(0,R_T)$ for $\omega$ close to $\omega_0$. By compactness of $\S^{n-1}$, a finite covering gives an $\varepsilon_0$ uniform in $\omega$; the corresponding local constructions agree by uniqueness.

Since the characteristic ODE depends smoothly on the parameters $\varepsilon$ and $\omega$, the corresponding flow, and hence the constructed solution $\psi$, depend smoothly on them as a
$C^\infty(B(0,R_T))$--valued map. Therefore, for every $N\ge0$,
\be{psi_exp}
\psi(x,\eps)=\sum_{j=0}^N \varepsilon^j\psi_j(x)+\varepsilon^{N+1}r_N(x,\varepsilon),
\ee
(depending smoothly on $\omega$ as well), with $r_N$ smooth, where
\[
\psi_j(x)=\frac1{j!}\,\partial_\varepsilon^j\psi(x,\eps)\big|_{\varepsilon=0}.
\]
Substituting this expansion into \eqref{eik2} and comparing coefficients of
$\varepsilon^n$ yields \eqref{phi_rec_short}. Since
\[
\psi|_{\varepsilon=0}(x)=x\cdot\omega,
\]
we get $\psi_0=x\cdot\omega$. The incoming condition in \eqref{eik2} implies
\[
\psi_j(x)=0
\qquad \text{for } x\in B(0,R_T),\quad x\cdot\omega<-R,\quad j\ge1.
\]
This completes the proof.
\end{proof}


Observe that in particular, in the coordinates $(s,y)$ for which $x=y+s\omega$,  $y\cdot\omega=0$,
\be{psi_1}
\psi_1(y,s)
=
-\frac12\int_{-\infty}^s \sigma_{\mathfrak{m}}(y+\tau\omega)\,d\tau.
\ee


\subsection{Deriving the transport equations} \ 

The initial conditions imply easily 
\be{ICa0}
a|_{t=0}= \chi(x).
\ee
Moreover, $a_0= \chi(x-t\omega)$ for $t\in[0,t_0]$. The subsequent terms $a_j$, $j\ge1$ are obtained by solving the standard transport equations (when $\mathfrak{m}=0$) with sources containing the previously computed ones, and they all have supports in $\supp\chi+t\omega$ for $t\in[0,t_0]$. In particular, $a\mathfrak{m}=0$ for such $t$, i.e., the interaction with the memory has not started yet. We anticipate that this support condition would be preserved along the orbits, perhaps perturbed by an $o(1)$ error as $h\to0$.  With this in mind, we define a relatively open set $U\subset [0,T]\times\R^n$ so that 
\begin{equation}   \label{suppU}
U \coloneq \big\{(t,x)|\; x\in \supp\chi+B(0,t_0/2)+ t\omega, \; t\in[0,T]\big\},
\end{equation}
We assume below that   all amplitude terms are supported in $U$ for $0<h\ll1$.  
Eventually, we will show that this assumption is consistent with the parametrix we construct.  The reason for this is to make sure that the memory integral does not integrate all the way to $s=t$ (or even beyond that).

We compute 
\begin{align}   \label{ephi1a}
e^{-\frac{\i \lambda }{h}\phi}(\partial_t^2-\Delta_{\sigma_\eps } ) e^{\frac{\i \lambda }{h}\phi}&= 
 \lambda^2h^{-2}  (-(\partial_t\phi)^2+ \sigma_\eps |\nabla\phi|^2)\\
& \qquad + \frac{2\i  \lambda }h (\partial_t\phi \partial_t - \sigma_\eps \nabla\phi \cdot\nabla  ) 
+ \frac{\i \lambda }{h} (\partial_t^2- \Delta_{\sigma_\eps }) \phi
+ (\partial_t^2-\Delta_{\sigma_\eps } ).
\end{align}
This explains the choice of the eikonal equation \r{eik1}, and assuming that $\phi$ solves it, the above simplifies to
\begin{equation}   \label{ephi2a}
e^{-\frac{\i \lambda }{h}\phi}(\partial_t^2-\Delta_{\sigma_\eps } ) e^{\frac{\i \lambda }{h}\phi}= 
- \frac{2\i \lambda }h ( \partial_t + \sigma_\eps \nabla\psi \cdot\nabla  ) 
-\frac{\i \lambda }{h} \Delta_{\sigma_\eps } \psi
+ (\partial_t^2-\Delta_{\sigma_\eps } ).
\end{equation}
We need to expand $e^{-\frac{\i \lambda }{h}\phi} M e^{\frac{\i \lambda }{h}\phi}$ next. We have
\begin{align}
  e^{-\frac{\i \lambda }{h}\phi} M e^{\frac{\i \lambda }{h}\phi} a & = e^{-\frac{\i  \lambda }{h}\psi} e^{\frac{\i  \lambda }{h}t} \nabla \cdot \int_0^\infty \mathfrak{m}(s,x) e^{-\frac{\i  \lambda }{h}  (t-s )} \nabla e^{\frac{\i  \lambda }{h} \psi} a(t-s,x,h)\,\d s \\ \label{Ih}
&=   e^{-\frac{\i  \lambda }{h}\psi}  \nabla \cdot \int_0^\infty \mathfrak{m}(s,x) e^{\frac{\i  \lambda }{h} s} \nabla e^{\frac{\i \lambda }{h} \psi} a(t-s,x,h)\,\d s .
\end{align}
Set 
\be{Ih0}
(I_h b)(t,x):=\int_0^\infty \mathfrak{m}(s,x)e^{\frac{\i \lambda  s}{h}}\,b(t-s,x,h)\,\d s .
\ee
Then
\[
e^{-\frac{\i \lambda }{h}\phi} M e^{\frac{\i \lambda }{h}\phi} a
=
e^{-\frac{\i \lambda }{h}\psi}\nabla\cdot\!\left(
e^{\frac{\i \lambda }{h}\psi}
\Big(\frac{\i \lambda }{h}\nabla\psi\, I_h a+I_h(\nabla a)\Big)\right),
\]
hence
\[
e^{-\frac{\i \lambda }{h}\phi} M e^{\frac{\i \lambda }{h}\phi} a
=
\Big(\nabla+\frac{\i \lambda }{h}\nabla\psi\Big)\cdot
\Big(\frac{\i \lambda }{h}\nabla\psi\, I_h a+I_h(\nabla a)\Big).
\]

We need first to understand the asymptotic behavior of \r{Ih0},  
where $b$ is smooth, supported in $U$, and bounded uniformly in $h$ together with all of its derivatives. By \r{supp} and \r{suppU}, $\mathfrak m(s,x)b(t-s,x,h)=0$ whenever $t-s\le t_0/2$. Thus, as a function of $s$, this product is supported in $[0,t)$ and has no singularity at $s=t$. Extended as zero outside $[0,t]$, it can be considered as the \sc\ Fourier transform (up to a change of the sign in the exponent). Its only singularity is at $s=0$, where it is of the kind $\sim m_0(x)s^{p-1}$.

\begin{lemma}\label{lemma_1}
Let $I_h$ be as in \r{Ih0}, 
where $b$ is smooth, supported in $U$, and bounded uniformly in $h$ together with all of its derivatives. 
Then $I_h b$ has an asymptotic expansion in powers of $(h/\lambda)^{p+k}$, $k=0,1,2,\dots$, with principal term
\[
I_h b = C_p m_0(x) b(t,x,h)  (h/ \lambda )^p   + O(h^{p+1})\quad \text{in $C^k$},
\] 
for every $k$. 
More generally, 
\[
I_h b\sim
\sum_{j\ge0} 
C_{p+j}(h/ \lambda )^{p+j}\,\frac1{j!}
\sum_{\ell=0}^j
\binom{j}{\ell}
m_\ell (x)\,
(-\partial_t)^{j-\ell}b(t,x,h)  \quad \text{in $C^k$}. 
\]
\end{lemma}

\begin{proof} 
It is enough to consider $k=0$.	Write 
\begin{align}
b(t-s,x) &=\sum_{k=0}^N \frac{ b^{(k)}(t,x)}{k!}(-s)^k + R_N'(t,s,x)s^{N+1}, \\
  m(s,x) &= \sum_{l=0}^N \frac{ m_l(x)}{l!} s^l + R_N''(s,x)s^{N+1}.
\end{align}
Leaving the remainders aside for a moment, we are left with the analysis of integrals of the type
\[
J_h^{(j)}  := \int_0^\infty  e^{\i \lambda  s /h}s^{j+p-1}\chi(s) \,\d s,
\]
$j=k+l=0,1,2,\dots$, where $\chi(s)=1$ near $s=0$, and it is compactly supported.   
The asymptotic is the same as that of 
\[
\int   e^{\i s \lambda  /h}s_+^{j+p-1} \,\d s = C_{p+j}(h/\lambda)^{p+j}, \quad 0<h\ll1, \;0<\lambda,
\]
see the proof of Lemma~\ref{lemma_m_exp}. The remainders are handled by the same argument but $\chi$ would depend smoothly on $(t,x)$ now. 
\end{proof}

\begin{remark}
	The expansion of $I_h b$ can be written down in the following compact form
\be{Ih1}
I_h b = \mu(\lambda/h -D_t)b	 ,
\ee
where $D_t=-i\partial_t$. Here, we formally consider $\lambda$ and $D_t$ as fixed parameters, and $h$ as a small one. Then we write
\[
(\lambda/h-D_t)^{-p-j} = (h/\lambda)^{p+j} (1-\lambda^{-1}hD_t)^{-p-j},
\]
and use the binomial series for the latter term to make sense of \r{Ih1} using Lemma~\ref{lemma_m_exp}. In fact, \r{Ih1} can also be derived, at least formally, through Fourier transforming $I_h b$. 
\end{remark}

By Lemma~\ref{lemma_1}, 
\begin{align}
I_h a&= C_p (h/ \lambda )^p m_0 a
+
C_{p+1}(h/ \lambda )^{p+1}(m_1 a-m_0\partial_t a)
+O(h^{p+2}),\\
I_h(\nabla a) &=
C_p (h/ \lambda )^p m_0 \nabla a
+ C_{p+1}(h/ \lambda )^{p+1}(m_1\nabla a-m_0\partial_t\nabla a)
+O(h^{p+2}).
\end{align}
Therefore, see \r{Ih},
\begin{align}
\frac{\i \lambda }{h}\nabla\psi\,I_h a+I_h(\nabla a) &=
\i C_p (h/ \lambda )^{p-1}m_0 a\,\nabla\psi\\
&\qquad +
C_p (h/ \lambda )^p m_0\nabla a
+
\i C_{p+1}(h/ \lambda )^p(m_1a-m_0\partial_t a)\nabla\psi
+O(h^{p+1}),
\end{align}
and hence
\begin{align} 
e^{-\frac{\i \lambda }{h}\phi} M e^{\frac{\i \lambda }{h}\phi} a
&=
- C_p (h/ \lambda )^{p-2} m_0 |\nabla\psi|^2 a \\
&\quad
+ \i C_p (h/ \lambda )^{p-1}
\Big(
2m_0\,\nabla\psi\cdot\nabla a
+a\,\nabla m_0\cdot\nabla\psi
+m_0 a\,\Delta\psi
\Big) \\
&\quad
- C_{p+1} (h/ \lambda )^{p-1} |\nabla\psi|^2
\big(m_1a-m_0\partial_t a\big)
+O(h^p). \label{ma}
\end{align}
Since $C_{p+1}=\i p\,C_p$, this can also be written as
\vskip -25pt
\begin{align}
h^{1-p}e^{-\frac{\i\lambda }{h}\phi} M e^{\frac{\i\lambda}{h}\phi} a
&=- C_p h^{-1}\lambda^{2-p} m_0 |\nabla\psi|^2 a \\
&\quad \label{eMe}
+ \i C_p \lambda^{1-p} 
\Big(
2m_0\,\nabla\psi\cdot\nabla a
+a\,\nabla m_0\cdot\nabla\psi
+m_0 a\,\Delta\psi
-p\,|\nabla\psi|^2(m_1 a-m_0\partial_t a)
\Big)\\
&\quad 
+O(h).
\end{align}
The $O(h)$ estimate above holds in every $C^k$, and it is uniformly bounded as long as $a$ is bounded in $C^{k+2}$. 
We combine this with \r{ephi2a} to expand 
\[
e^{-\frac{\i\lambda }{h}\phi}(\partial_t^2-\Delta_{\sigma_\eps}  + h^{1-p}M ) e^{\frac{\i\lambda }{h}\phi}
\]
with $a$ as in \r{ansatz}. We treat $\eps$ and $h$ as independent parameters for now. We get the first transport equation
\be{Tr1}
 \mathcal T a_0=0, \quad \mathcal T a:=( \partial_t + \sigma_\eps \nabla\psi \cdot\nabla  )a + \frac12 (\Delta_{\sigma_\eps}\psi)a -\frac12 \i C_p  \lambda^{1-p} m_0  \sigma_\eps^{-1}a 
\ee
with the initial condition $a_0|_{s=0}=\chi(x)$, see \r{ICa0}. For future reference, we note that when $a_0$ solves \r{Tr1}, we have $P_h  e^{\i\lambda\phi/h}a_0= O(1)$ in any semiclassical $C^k_h$ norm.

For fixed $\omega$, let $\gamma_{y,\omega}$ be the integral curve of
$\sigma_\eps\nabla\psi$ satisfying
\[
\dot\gamma_{y,\omega}(s)
=
\sigma_\eps(\gamma_{y,\omega}(s))
\nabla\psi(\gamma_{y,\omega}(s)),
\qquad \gamma_{y,\omega}(0)=y.
\]
We introduce coordinates $(s,y)$ by
\[
(t,x)=(s,\gamma_{y,\omega}(s)).
\]
For $0<\eps\ll1$, this map is a diffeomorphism on any fixed compact set
under consideration. In these coordinates,
\[
\partial_t+\sigma_\eps\nabla\psi\cdot\nabla_x = \partial_s, \quad \psi =y\cdot\omega+s. 
\]
When $\eps=0$, we have $x=\gamma_{y,\omega}(s) = y+s\omega$, and $ \partial_t + \sigma_\eps \nabla\psi \cdot\nabla = \partial_t +  \omega\cdot\nabla = \partial_s$. The solution of \r{Tr1}, in the $(s,y)$ coordinates, is
\begin{align}\label{a0}
a_0(s,\gamma_{y,\omega}(s)) &= \chi(y) \exp\Big(
-\frac12\int_0^s \Delta_{\sigma_\eps}\psi(\gamma_{y,\omega}(\tau))\,\d\tau\\
&\qquad +\frac{\i}{2}C_p\lambda^{1-p} \int_0^s m_0(\gamma_{y,\omega}(\tau))\sigma_\eps^{-1}(\gamma_{y,\omega}(\tau))\,\d\tau
\Big).
\end{align}
The particular case $\eps=0$, which approximates \r{a0} up to $O(\eps)$, is
\be{a001}
a_0 (t,x) = \exp \left( 
\frac{\i}{2}\,  C_p \lambda^{1-p}
\int_0^t
m_0(\,x-\tau \omega)\,d\tau 
\right)\chi(x-t\omega), \quad\eps=0,
\ee
written in the $(t,x)$ coordinates. 

Equating the $h^0$ powers, we get the second transport equation
\begin{align}
\mathcal T a_1& =  -\frac{\i \lambda^{-1}}{2}(\partial_t^2-\Delta_{\sigma_\eps})a_0 
+\frac{C_p}{2}\lambda^{-p} \Big(
2m_0\nabla\psi\cdot\nabla a_0
+a_0\nabla m_0\cdot\nabla\psi
+m_0 a_0\,\Delta\psi \\
&\qquad \label{Tr2}
-p\,\sigma_\eps^{-1}(m_1a_0-m_0\partial_t a_0)
\Big)
\end{align}
with an initial condition $a_1|_{s=0}=0$.  

Similarly, for $a_j$, $j\ge2$, we get 
\begin{equation}   \label{Tr3}
\mathcal T a_j = S[a_0,\dots,a_{j-1}],\quad a_j|_{s=0}=0,
\end{equation}
where $S$ depends on finitely many derivatives of $a_0,\dots,a_{j-1}$, on $\psi$, on $m_0,\dots,m_j$, and on $\lambda$ and $\eps$. Then $u_N := e^{\i\lambda\phi /h}\sum_{j=0}^{N}h^j a_j$ satisfies $P_hu_N= O(h^N)$ in every $C_h^k$, uniformly in $\lambda$ in \r{lambda}, with the incoming initial conditions induced by \r{IC1}.

\subsection{Expansion of each amplitude in powers of $\eps$}
Each amplitude term $a_k$ depends smoothly on $\eps$ through $\sigma_\eps$ and $\psi$. It therefore has a Taylor expansion in $\eps$ about $\eps=0$. Direct calculations using \r{psi_exp} show that, upon expanding the leading amplitude as
\[
a_0(x,\eps)\sim \sum_{j\ge0}\varepsilon^j a_0^{(j)}(x),
\]
we get the equation
\begin{equation}   \label{a0012}
(\partial_t+ \omega\cdot\nabla) a_0^{(0)} -\frac12 \i\lambda^{1-p} C_p m_0 a_0^{(0)}=0, \quad a_0^{(0)}|_{s=0}=\chi(x),
\end{equation}
whose solution is \r{a001}, and similar but nonhomogeneous ODEs for $a_0^{(j)}$, $j\ge1$, with the same transport operator on the left, a right-hand side depending on the previous coefficients, and zero initial conditions.

We formulate the main theorem of this section.

\begin{theorem}[The asymptotic expansion] \label{thm_exp} 
Assume that $\mathfrak{m}$ satisfies \eqref{m0}, \eqref{supp},
\eqref{m-monotone}, and \r{m-int-k} for every $k$. 
Fix $T>0$. Then problem \r{1b} with initial conditions \r{IC1} has a unique solution $u=u_\text{\rm in}+u_\text{\rm out}$ in $[0,T]\times\R^n$ for $0<h\ll1$. Here $u_\text{\rm out}$, determined up to $O(h^\infty)$ in any $C_h^k$, is the standard geometric optics solution corresponding to $\mathfrak{m}=0$ and propagating away from $\supp_x\mathfrak{m}$, while $u_\text{\rm in}$ is of the form \r{ansatz}, where $\phi(t,x,\eps)$ is as in \r{phase}, $\eps=h^{1-p}$, and $a_j=a_j(t,x,\eps)$ solve the transport equations \r{Tr1}, \r{Tr2}, and \r{Tr3}.
\end{theorem}

The expansions in \r{ansatz} are understood in $C^k$ for every $k$ after factoring out the oscillatory exponential, or equivalently in the semiclassical $C_h^k$ norms for $u_\text{\rm in}$, uniformly for $\lambda$ satisfying \r{lambda}.

The proof is provided in section~\ref{sec_just}. 

\subsection{Analysis of the asymptotic expansion} \label{sec_analysis}

By Theorem~\ref{thm_exp}, 
\[
u_{\rm in}(t,x) = e^{\i \lambda \phi(t,x,\eps)/h}a_0(t,x,\eps)+ O(h),
\]
where for $a_0$, we have \r{a0} in characteristic coordinates. In these ($\eps$-dependent) coordinates, $\phi = -t+s+y\cdot\omega$. This estimate is preserved under differentiation by powers of $h\p_t$ and $h\p_x$, and one can include lower-order terms as well. We can draw the following conclusions about the effect of the memory term:

\begin{itemize}
\item The geometry is modified from the flat geometry (speed one) to a curved one with the faster speed $\sqrt{\sigma_\eps} = 1+\frac12 h^{1-p}\sigma_{\mathfrak{m}}(x) + O(h^{2(1-p)})$; see \r{sigma_eps}.

\item There is an attenuation caused by $\Im C_p$ in \r{a0}, and the corresponding accumulated exponent is $\frac12\Im C_p \lambda^{1-p} (Bm_0)(s)$, where  $(Bm_0)(s) $ is the beam ray transform of $m_0$ appearing in \r{a0}.  
 
\item There is also a phase shift caused by $\Re C_p$ in \r{a0}.
\end{itemize}

The way we assess the phase speed is to write the real part of the phase, combining $\lambda \phi/h$ and the phase of $a_0$, as
\[
\Re\Phi:= \lambda (-t+s+y\cdot\omega)/h +\frac12 \Re C_p \lambda^{1-p}(Bm_0)(s), 
\]
where $(Bm_0)(s) $ is the beam ray transform of $m_0$ appearing in \r{a0}. To find the speed at which the constant levels move, we compute $\d s/\d t$ in $\Re\Phi=\text{const.} $ by implicit differentiation 
\begin{align}
\frac{\d s}{\d t} &= 1- \frac12 h (\Re C_p)\lambda^{-p}(Bm_0)'(s)+O(h^2)\\
&=  1-  \frac12 h\Gamma(p)\cos \frac{p\pi}2   \lambda^{-p} m_0(x)\sigma_\eps^{-1}(x) +O(h^2). 
\label{ph_speed}
\end{align}
The second term on the right-hand side of \r{ph_speed} is the correction to the phase speed of the wave packet relative to the speed $1$, which is the speed in the metric $\sigma_\eps^{-1}\d x$. In Euclidean terms, this term represents the principal part of the \textit{relative} correction of the speed $\sqrt{\sigma_\eps}$. Note  the dependence on the frequency $\lambda$, which shows that this correction causes frequency dispersion; higher frequencies propagate faster. 

Similar calculations allow us to estimate the phase shift:
\[
\delta s
=
-\frac12 h\,\Gamma(p)\cos \frac{\pi p}{2} \,
\lambda^{-p}(Bm_0)(s_0)
+O(h^2).
\]
It might be useful to look at the phase shift $\delta s/(2\pi h/\lambda)$ relative to the wavelength. That quantity is proportional to $\lambda^{1-p}$.

The dependence on $p$  deserves attention as well. When $p>0$ is close to $p=0$, the phase delay gets large because of the Gamma function; and when $p=1$ (no singularity), it is zero. The attenuation is equal to  $\frac12\Gamma(p)\sin(p\pi/2)$ times the beam transform of $m_0$. This coefficient is $\pi/4$ when $p=0$, i.e., when the memory kernel approaches the critical case. It is $1/2$ when $p=1$ (no singularity). 

\begin{figure}[tbp] 
\centering
\includegraphics[scale=0.5]{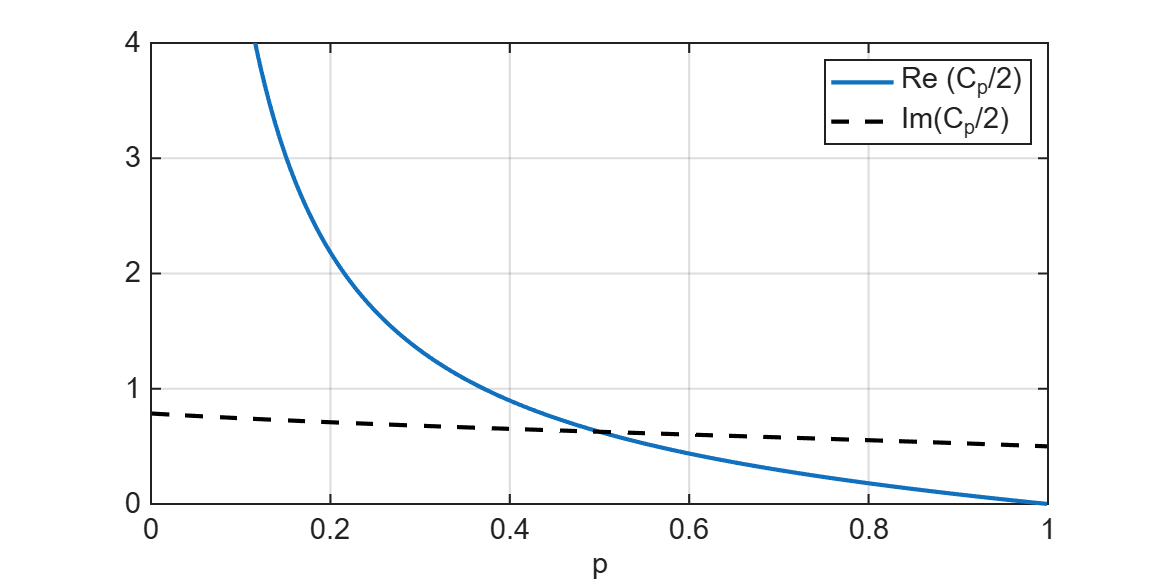}
\caption{The attenuation rate, $\Im C_p/2$, and the phase delay rate, $\Re C_p/2$, as functions of $p\in(0,1)$. The phase delay increases linearly in $t$ (when $m_0=\text{const.}$, for example), at the rate $\Re C_p/2$, while the amplitude decreases exponentially in $t$ at the rate $\Im C_p/2$.}
\end{figure}

\subsection{Approximate PDE}
For $0<h\ll1$, one can replace the non-local  integro-differential equation \r{1b} by the following damped \textit{differential} wave equation 
\begin{equation}   \label{approx_PDE}
u_{tt}-\i C_p\lambda^{1-p} m_0(x)\sigma_\eps^{-1} u_t-\Delta_{\sigma_\eps} u=0 
\end{equation}
with the first two initial conditions in \r{IC1}. 
Indeed, the standard geometric optics for \r{approx_PDE} leads to the same ansatz with the same leading transport equation, \r{Tr1}. The ansatz can be justified by the well-posedness of the damped wave equation. This yields the following.

\begin{theorem}[Leading comparison with the local approximate equation] \label{thm_appr} 
Fix $0<\lambda_{\min}\le \lambda\le \lambda_{\max}<\infty$, $T>0$, and let
$u_{\rm in}$ denote the incoming branch of the solution of \eqref{1b},
\eqref{IC1}. Let $u_{\rm app}$ denote the corresponding incoming
geometric-optics solution of \eqref{approx_PDE}, with the same incoming
initial amplitude and phase. Then, for every $k\ge 0$,
\[
u_{\rm in}(t,x)-u_{\rm app}(t,x)
=
e^{i\lambda\phi(t,x,\eps)/h} r(t,x,h,\eps),
\]
where
\[
\|r\|_{C^k([0,T]\times\R^n)}
\le C_k h.
\]
 Equivalently,
\[
u_{\rm in}=u_{\rm app}+O(h)
\]
in the semiclassical $C^k_h$ norm.
\end{theorem}

\begin{proof}
As mentioned above, the standard geometric optics for \r{approx_PDE} yields the same kind of expansion, with the same phase, and with the same leading amplitude $a_0$. If $u_{\textrm{in},N}$ and $u_{\textrm{app},N}$ are the corresponding truncated asymptotic expansions, then for every $k$, $e^{-\i\lambda\phi/h}(u_{\textrm{in},N}-u_{\textrm{app},N})=O(h)$ in $C^k$ for $N\gg1$. Justification of the asymptotic expansion for \r{approx_PDE} is even easier, based on standard energy estimates. Then since the claim is true for the truncated asymptotic expansions, it is true for the exact solutions by Theorem~\ref{thm_well_posed} and its version for \r{approx_PDE}. 
\end{proof}

This theorem could have applications to time reversal in thermoacoustic tomography. One of the techniques there is to send the observed waves back to $t=0$ by reversing the time propagation; see, e.g., \cite{SU-thermo} and the references therein. The original memory wave equation \r{1b}, \r{IC1} cannot be solved backward in time because that would require knowledge of the future history with time inverted. On the other hand, \r{approx_PDE} can be solved backward. It gives a solution semigroup that increases exponentially in time, but we solve it only over a finite time interval. Time reversal for the damped wave equation has been implemented in \cite{Andrew13, AcostaM_16}. Their case corresponds formally to $\Re C_p=0$, $\Im C_p>0$, but the nonzero $\Re C_p$ creates only a phase shift and would not be an obstacle. 

Another benefit of using \r{approx_PDE} is that solving a PDE numerically with finite differences is much faster than solving the integro-differential equation \r{1b}. For the latter, computing $u(t+\Delta t,x)$ with a second-order finite-difference scheme requires the history not only of $u(t,x)$ and $u(t-\Delta t,x)$, but also at earlier times until $\mathfrak{m}$ becomes sufficiently small. This requires storing all those values and integrating them at every time step. Moreover, this integration must be performed carefully near $s=0$ because of the singularity there. There are no such issues with \r{approx_PDE}. 

A downside with this approach is that it works for monochromatic signals only, i.e., for $\lambda$ fixed. A broadband signal needs to be decomposed into monochromatic ones if we want to use \r{approx_PDE}. 
 
\subsection{The $p=1$ case} Assume $p=1$ now. Then the problem has its original form since for the prefactor we have $h^{1-p}=1$. By \r{m0}, $\mathfrak{m}(s,x) = m(s,x) \in C^\infty([0,\infty)\times\R^n)$. Lemma~\ref{lemma_1} still holds. We have $\eps=1$ in \r{varepsilon}, so it is not a small parameter anymore. The phase $\phi$ solves the eikonal equation \r{eik1} with $\sigma_1= 1+\sigma_{\mathfrak{m}}$; it is of the kind \r{phase} with $\psi$ solving \r{eik2}.  We assume that this equation is solvable in $B(0,R_T)$. The first transport equation \r{Tr1} has the same form
\be{Tr1a}
 \mathcal T a_0=0, \quad \mathcal T a:=( \partial_t + \sigma_1 \nabla\psi \cdot\nabla  )a + \frac12 (\Delta_{\sigma_1}\psi)a +\frac12  m_0\sigma_1^{-1}a,
\ee
where we used the fact that $C_p=i$ for $p=1$. 
Compared to the standard leading transport equation for $\Delta_\sigma$, we have the extra term  $\frac12  m_0\sigma_1^{-1}a$, which is positive, and as such causes damping. The expression \r{a0} for the leading amplitude reduces to 
\begin{align}\label{a01}
a_0(s,\gamma_{y,\omega}(s)) &= \chi(y) \exp\Big(
-\frac12\int_0^s \Delta_{\sigma_1}\psi(\gamma_{y,\omega}(\tau),1)\,\d\tau\\
&\qquad -\frac{1}{2}  \int_0^s m_0(\gamma_{y,\omega}(\tau))\sigma_1^{-1}(\gamma_{y,\omega}(\tau))\,\d\tau
\Big),
\end{align}
where $\psi(x,1)=\psi(x,\eps)$ for $\eps=1$.  
The effect of the memory therefore is to change the effective speed from $c=1$ to $c=\sqrt{\sigma_1}$, and to introduce attenuation at a rate  $\frac12  m_0\sigma_1^{-1}$.  The approximate PDE then would be
\begin{equation}   \label{approx_PDE3}
u_{tt}+  {m}_0(x)\sigma_1^{-1} u_t-\Delta_{\sigma_1} u=0 ,
\end{equation}
where $\sigma_1=1+\sigma_{\mathfrak{m}}$ as above. Observe that there is no $h$ or $\lambda$ in the PDE, the speed is changed from one to $\sqrt{1+\sigma_{\mathfrak{m}}}$, and there is a damping coefficient $m_0(x)$. There is no leading phase shift.

\subsection{A semiclassical approach} 
Set $\lambda =|\xi|$, $\omega=\xi/|\xi|$, where $\lambda$ is as in \r{lambda}. 
The incoming prehistory now is
\[
u|_{t\le0}= \chi\big(x-t\xi/|\xi|\big) e^{\i (-t|\xi|+x\cdot\xi)/h}. 
\]
The solution is given by Theorem~\ref{thm_exp} with the above change of variables. 

A more general version of this, involving a broadband signal, is to solve \r{1b} with initial conditions
\be{ICa}
u|_{t=0} = f_1(x,h), \quad u_t|_{t=0} = f_2(x,h),
\ee 
where $f_1$ and $f_2$ are compactly supported outside $\overline{B(0,R)}$ and semiclassically tempered, with $\WFH(f_1)$ and $\WFH(f_2)$ contained in $\R^n\setminus \overline{B(0,R)}\times \{\lambda_\textrm{min}< |\xi|< \lambda_\textrm{max}\}$.  

Since the wave equation has a positive and a negative speed, the solution naturally splits into four terms involving combinations of both speeds and both initial conditions; see, e.g., \cite{SU-thermo_brain}. We get 
\begin{align}   \label{u_par}
u(t,x) &= (2\pi h)^{-n}  \sum_{\rho=\pm }\int_{\lambda_\textrm{min}} ^{\lambda_\textrm{max}}\int_{\S^{n-1}} e^{\i \lambda \phi_{\rho}(t,x ,\omega,\eps)/h} \Big( a_{1,\rho}(t,x,\omega, \lambda, h,\eps)   (\mathcal{F}_hf_1)(\lambda\omega)\\
& \qquad + \frac{h}{\lambda} a_{2,\rho}(t,x,\omega, \lambda, h,\eps)   (\mathcal{F}_hf_2)(\lambda\omega)   \Big)\lambda^{n-1} \d  \lambda\, \d\omega+O(h^\infty),
\end{align}
where $\mathcal F_h$ is the \sc\ Fourier transform. 
Here, we changed the notation for the phases and for the amplitudes a bit to indicate the dependence on $\omega$ and $\lambda$ as well. Next, 
$\phi_{\pm}=\mp t+\psi_\pm$, see \r{phase}, and for every $\omega$, only one of $\phi_{\pm}$ can differ from its trivial choice $\mp t+x\cdot\omega$ because, for $|x|>R$, at least one of the rays $[0,\infty)\ni s\mapsto x+s\omega$ does not hit $B(0,R)$. All amplitudes are of order $0$ in $h$ and have two-scale expansions as above. The initial conditions for their transport equations can be derived from \r{ICa} as in \cite{SU-thermo_brain}. Equivalently, the amplitudes are polyhomogeneous symbols in integral powers of $h$, depending smoothly on $\eps=h^{1-p}$ in the Fr\'echet symbol topology up to $\eps=0$ and smoothly on $\lambda$ as in \r{lambda}. As semiclassical symbols, they belong to the standard symbol classes since $t$- and $x$-derivatives do not change the order of the symbol estimates. Formula \r{u_par} can be considered as a \sc\ FIO depending smoothly on $\eps=h^{1-p}$, formally regarded as an independent parameter. Its canonical relation is therefore also $\eps$-dependent. 

Since the test for $\WFH(u)$ includes $h$-independent cutoffs, one can show that 
\[
\WFH(u(t,\cdot),u_t(t,\cdot) )= C_0(t)\circ \WFH(f_1,f_2),
\]
where $C_0(t)$ is the canonical relation corresponding to $\eps=0$. This does not, however, capture the fractional effects described earlier, such as the phase shift and the $\eps$-dependent changes in the geometry.

\section{Numerical examples} We illustrate some of the effects with a few numerical simulations. In all of them, $\sigma=1$. 

\subsection{Phase shift and attenuation} 
In Figure~\ref{fig_phase}, we plot as a solid curve a one-dimensional solution moving from left to right, with the initial wave packet labeled $t=0$. Snapshots of the solution are also plotted at $t=T/2$ and $t=T$, where $T>0$ is the final time. For comparison, we plot the ``free'' solution corresponding to $\mathfrak{m}=0$ as a dashed curve. That solution simply translates from left to right. Finally, we plot the solution called ``intermediate,'' obtained by retaining $\sigma_\eps$ in \r{1c} but dropping the memory integral. The memory kernel $\mathfrak{m}$ is supported in the middle, and the gray shading illustrates its density. 

We observe first that the intermediate solution moves faster than the free one because its speed is $\sqrt{\sigma_\eps}>1$. The memory solution, however, moves slower than the intermediate one but still faster than the free one. This illustrates the phase shift computed above. The memory solution also attenuates exponentially. The frequency $\lambda$ is fixed here, and $p=0.5$. 

\begin{figure}[bpp]
\includegraphics[scale=0.62]{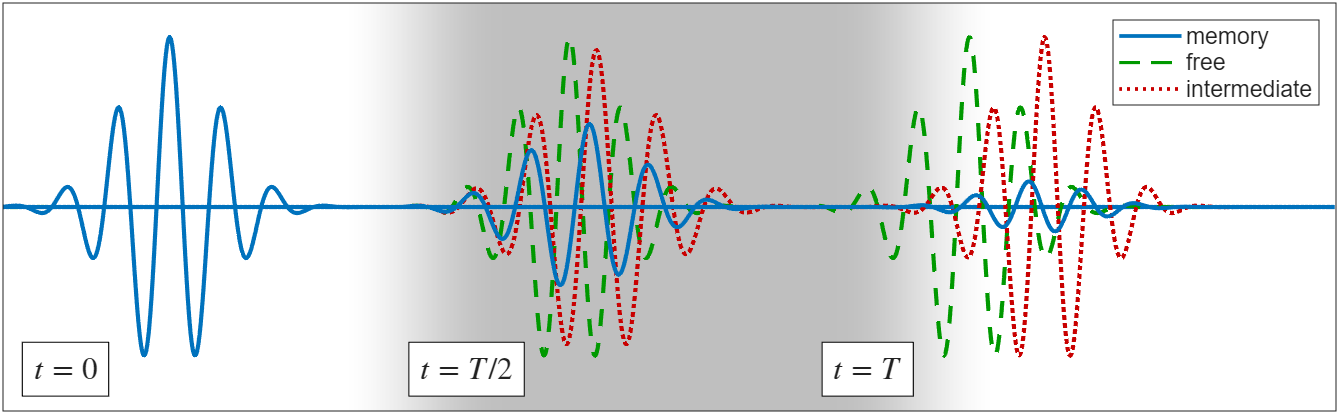}
\caption{$p=0.5$. Three solutions at $t=0, T/2, T$: (a) the memory solution as a solid curve; (b) the free one ($\mathfrak{m}=0$) dashed, and (c) the solution with $\sigma_\eps$ preserved but the memory integral in \r{1c} omitted, dotted. This illustrates the increase of the speed due to $\sigma_\eps$ and then the phase shift slowing down the memory solution compared to the dotted one.}  \label{fig_phase}
\end{figure}

\subsection{Dependence on the frequency $\lambda$ (dispersion)} 
In Figure~\ref{fig_phase_freq}, we demonstrate the dependence on $\lambda$. We plot the initial and the final states only, with $T$ slightly larger than in the previous example. We send two packets of frequencies having ratio $3$. In comparison, we plot the intermediate solution as in the previous example. The effect is subtle and it can be seen by the phase shift of the solid curve at $t=T$ compared to the dotted one. For the lower frequency packet, it is negative, a bit less than $\pi/2$ in absolute value as a fraction of the wavelength. For the higher frequency packet, it is still negative, a bit larger than $1.5$ times that in absolute value, which is consistent with being $(\lambda_2/\lambda_1)^{1-p} = \sqrt 3$ larger up to a higher order correction; see section~\ref{sec_analysis}. A direct comparison of the distance between the zeros or the peaks of both packets, at $t=0$ and $t=T$, shows a slight increase of the distance between the two. Another effect we observe is that the higher frequency attenuates faster. In Figure~\ref{fig_phase_freq2}, we plot two such packets independently, which demonstrates clearly that the net effect of the higher $\lambda$ is a faster speed.

\begin{figure}[tbp]
\includegraphics[scale=0.62]{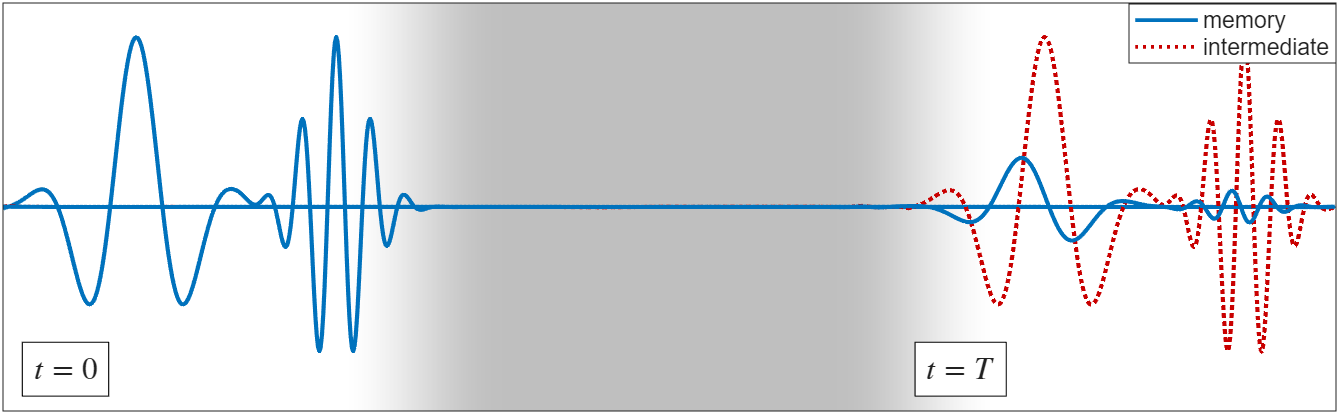}
\caption{$p=0.5$. Two packets at $t=0$ and $t=T$: the higher-frequency one moved a bit faster through the memory medium but experienced a negative phase shift larger in absolute value relative to the intermediate one, as a fraction of the wavelength. It attenuated more as well.}
 \label{fig_phase_freq}
\end{figure}

\begin{figure}[tbp]
\includegraphics[scale=0.62]{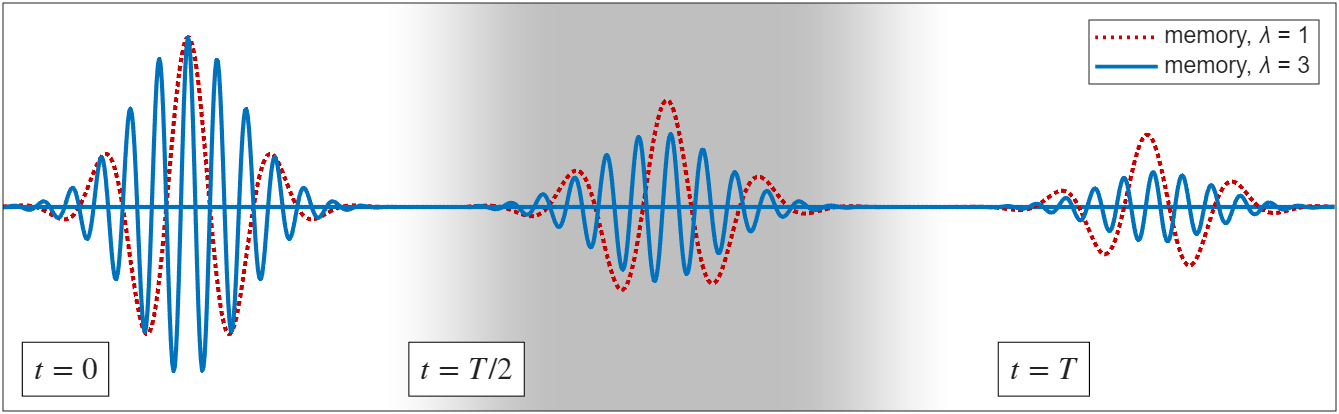}
\caption{$p=0.5$. Two packets at $t=0$, $t=T/2$, and $t=T$, plotted independently: the combined effect of the faster speed and the larger negative phase shift is still a faster speed of the higher-frequency packet.}
 \label{fig_phase_freq2}
\end{figure}

\subsection{2D propagation} 
In Figure~\ref{fig_2D}, we propagate an oscillatory plane wave in two dimensions. The $x$-support of the memory is in the middle of the computational domain. The plot clearly shows that the speed has changed from $1$ to $\sqrt{\sigma_\eps}$ in the middle and that the wave has attenuated. The phase shift can be seen by plotting a vertical cross-section through the middle, but then we are in the regime of the previous example. 

\begin{figure}[tbp]
\includegraphics[scale=0.32]{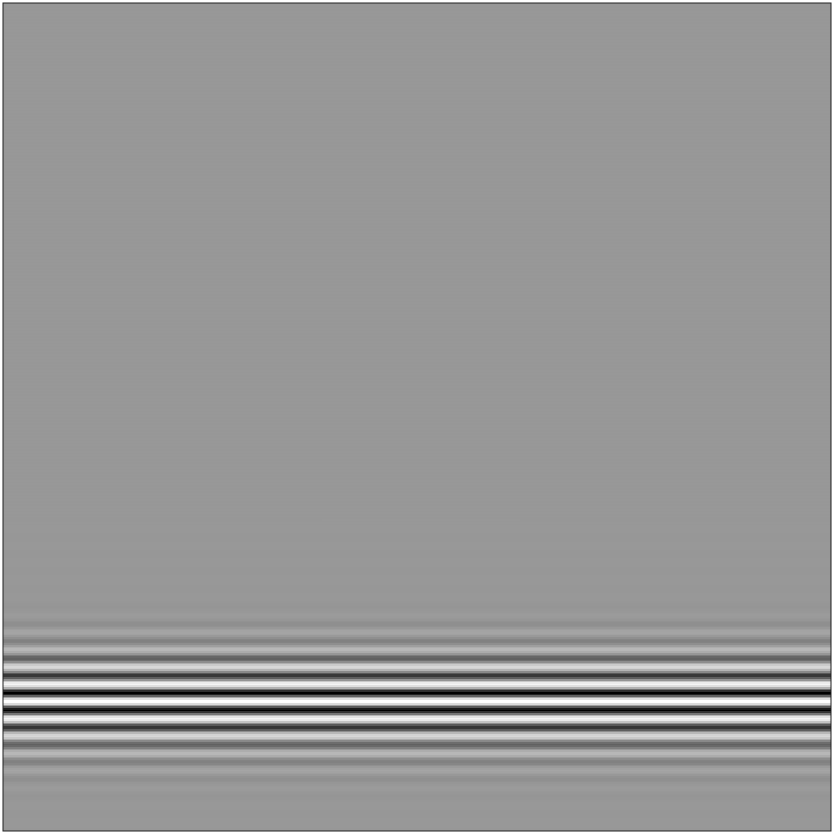}\hspace{40pt}
\includegraphics[scale=0.32]{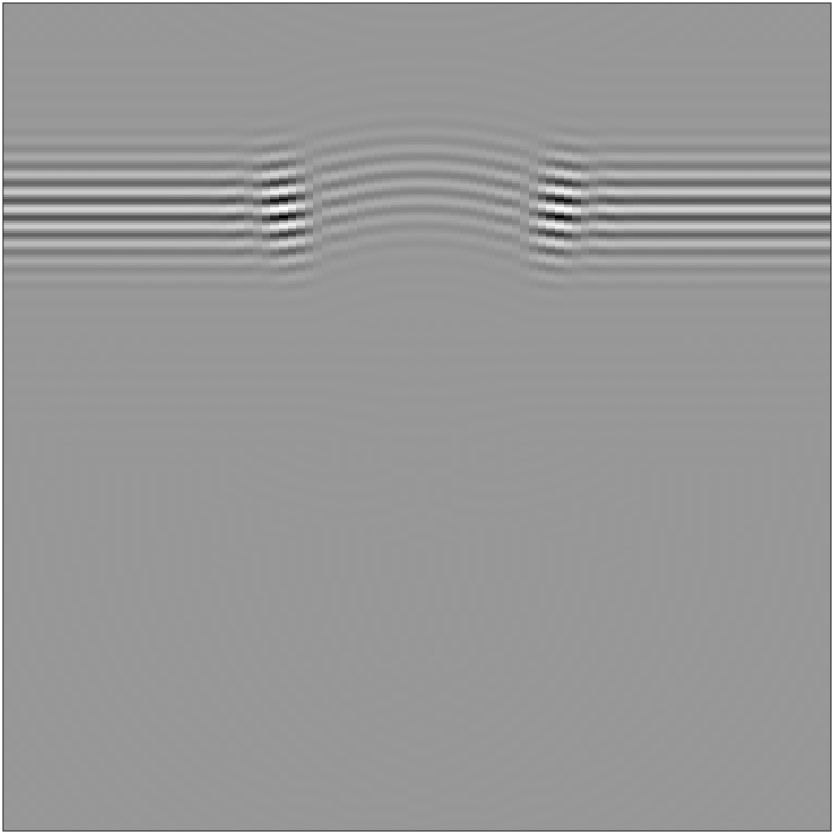}
\caption{$p=0.5$. A plane wave moving up. Left: the initial condition. Right: The solution at time $t=T$. The $x$-support of the memory is in the middle.}
 \label{fig_2D}
\end{figure}

\section{The Inverse Problem} \label{sec_IP}
\begin{theorem}[the inverse problem] \label{thm_IP} Let $0<p<1$. 
For every $\omega\in \S^{n-1}$, let $\chi\in C_0^\infty(\Omega)$ satisfy $\supp\chi\cap\overline{B(0,R)}=\emptyset$. Let $T>0$ be such that every point of $\overline{B(0,R)}$ lies on a ray $z+t\omega$, with $z\in\{\chi\ne0\}$ and $0<t<T$, and every such ray has exited $\overline{B(0,R)}$ by time $T$. Assume that $\Omega$ contains all the relevant incoming and outgoing ray segments up to time $T$. Then the family of corresponding solutions $u$ of \r{1b}, \r{IC1}, restricted to $[0,T]\times(\Omega\setminus B(0,R))$, for a fixed $\lambda$ in \r{lambda} and all $0<h\ll1$, uniquely determines $\sigma_{\mathfrak{m}}$ and the full jet $\{m_k(x)\}_{k\ge0}$ of $m(s,x)$ at $s=0$. 
\end{theorem}

\begin{remark}
We could have formulated the theorem with data the Dirichlet-to-Neumann map on $[0,T]\times \p\Omega$. Also, for each $\omega$, we need to know $u$ only along the rays in $ \Omega\setminus B(0,R)$ once they have exited $B(0,R)$. Next, for every $\omega$ we may have several beams through that $B(0,R)$ , not necessarily one wide one illuminating the whole ball. Finally, some restricted set of directions would be enough as long as the corresponding X-ray transform is injective on that set. 
\end{remark} 

\begin{proof}[Proof of Theorem~\ref{thm_IP}] 
Suppose that two memory kernels $\mathfrak{m}$ and $\tilde{\mathfrak{m}}$, satisfying the assumptions of the theorem, produce the same exterior measurements for every sufficiently small $h$, and denote the quantities corresponding to the latter kernel by tildes. Fix an observation point $x\in \Omega\setminus B(0,R)$ on an outgoing ray whose free incoming point $z$ lies in the interior of $\{\chi\ne0\}$. Then, by \r{ansatz}, 
\[
e^{\i \lambda \psi(x,\omega,\eps)/h}  a_0(t,x,\eps) = e^{\i \lambda \tilde \psi(x,\omega,\eps)/h}  \tilde a_0(t,x,\eps) + O(h),
\]
where we divided by $e^{-\i \lambda t/h }$ to pass from $\phi$ to $\psi$. 
Although the leading amplitude $a_0(t,x,\eps)$ may itself be complex, its phase has no explicit $h$-dependence when $\eps$ is regarded as an independent parameter. After setting $\eps=h^{1-p}$, it depends smoothly on $h^{1-p}$ and, in particular, contains no oscillations at the scale $1/h$; see \r{a0}.  For sufficiently small $\eps$, the incoming points of the corresponding characteristics for both media remain in this set; hence \r{a0} shows that $a_0$ and $\tilde a_0$ are bounded away from zero along the observed ray. Setting $r(h):=\psi(x,h^{1-p})-\tilde\psi(x,h^{1-p})$ and dividing the preceding equality by $e^{\i\lambda\tilde\psi/h}a_0$, we obtain $e^{\i\lambda r(h)/h}=\tilde a_0/a_0+O(h)$, whose right-hand side has a limit as $h\to0+$. Let $N$ be the smallest integer such that $N(1-p)\ge1$. Since $r(0)=0$ and $r$ is smooth in $\eps$, Taylor's theorem gives $r(h)=\sum_{j=1}^{N-1}c_jh^{j(1-p)}+O(h^{N(1-p)})$. If any $c_j\ne0$, taking the first such $j$ shows that $e^{\i\lambda r(h)/h}$ cannot have a limit. Consequently, all these coefficients vanish and $r(h)=O(h)$, i.e., $\psi(x,\eps)-\tilde\psi(x,\eps)=O(h)$.
Write the corresponding Euclidean line as $y+\mathbb R\omega$, with $y\in\omega^\perp$, and take the observation after the ray has exited $B(0,R)$. By Proposition~\ref{prop:phase_smooth} and \r{psi_1}, 
\[
\eps \int_{-\infty}^\infty (\sigma_{\mathfrak{m}} - \tilde \sigma_{\mathfrak{m}})  (y+\tau\omega)\, \d\tau +O(\eps^2) = O(h). 
\]
Dividing by $\eps$, we get
\[
 \int_{-\infty}^\infty (\sigma_{\mathfrak{m}} - \tilde \sigma_{\mathfrak{m}})  (y+\tau\omega)\, \d\tau  = O(h^{\min(p,1-p)}). 
\]
Letting $h\to 0+$, we get that the X-ray transform of $\sigma_{\mathfrak{m}} - \tilde \sigma_{\mathfrak{m}}$ along those rays vanishes. Varying $\chi$ and $\omega$, we get $\sigma_{\mathfrak{m}} = \tilde \sigma_{\mathfrak{m}}$. 

Once we have recovered $\sigma_{\mathfrak{m}}$, we know the phase $\phi(x,\omega,\eps)$. 
Then $a_0(t,x,\eps) = \tilde a_0(t,x,\eps)+O(h)$, where we indicated the dependence on $\eps$ as well. Since that dependence is smooth, we can take the limit $h\to 0+$; then $\eps\to0+$ as well, to conclude $a_0(t,x,0) = \tilde a_0(t,x,0)$. Since $\chi(x-t\omega)\ne0$, we may cancel this common factor in \r{a001}. Moreover, $\Re(\i C_p)=-\Gamma(p)\sin(\pi p/2)<0$, so the map $q\mapsto\exp(\frac{\i}{2}C_p\lambda^{1-p}q)$ is injective on $\mathbb R$. Hence $a_0(t,x,0)=\tilde a_0(t,x,0)$ implies equality of the corresponding ray integrals of $m_0$ and $\tilde m_0$. Once the ray has exited $B(0,R)$, these are their full X-ray transforms, and therefore $m_0=\tilde m_0$. Since $\sigma_{\mathfrak m}$ and $m_0$ are now known, \r{a0} determines $a_0(t,x,\eps)$ for every sufficiently small $\eps$. Taking the expansion of the equality of the observed solutions one order further, subtracting the common leading term, and dividing by $h$, we obtain $a_1-\tilde a_1=O(h)$. More generally, once $m_0,\ldots,m_{j-1}$ have been recovered, the corresponding amplitudes $a_0,\ldots,a_{j-1}$ agree; subtracting these common terms from an expansion through order $h^{j+1}$ and dividing by $h^j$ gives $a_j-\tilde a_j=O(h)$.

Let $b_1=a_1-\tilde a_1$. Since $\sigma_{\mathfrak m}$ and $m_0$ have already been recovered, the two transport operators and the leading amplitudes agree. Subtracting \r{Tr2} therefore gives
\[
\mathcal T b_1=-\frac{pC_p}{2}\lambda^{-p}\sigma_\eps^{-1}a_0(m_1-\tilde m_1).
\]
Along a characteristic $\gamma$, the identity $\mathcal T a_0=0$ then yields
\[
\frac{\d}{\d s}\left(\frac{b_1}{a_0}\right)(s,\gamma(s))
=-\frac{pC_p}{2}\lambda^{-p}\sigma_\eps^{-1}(\gamma(s))(m_1-\tilde m_1)(\gamma(s)).
\]
The incoming value of $b_1$ is zero, whereas its outgoing value is $O(h)$. Letting $h\to0+$, and hence $\eps\to0+$, we obtain the ordinary X-ray transform of $m_1-\tilde m_1$ along the corresponding Euclidean line. Its injectivity gives $m_1=\tilde m_1$.

Next, assume that $\sigma_{\mathfrak m}$ and $m_0,\ldots,m_{j-1}$ have already been recovered, where $j\ge2$. The full expansion in Lemma~\ref{lemma_1} shows that the transport equation for $a_j$ has the form $\mathcal T a_j=F_j+\kappa_j\sigma_\eps^{-1}a_0m_j$, where $F_j$ is determined by the previously recovered coefficients and
\[
\kappa_j=\frac{\i C_{p+j}}{2j!}\lambda^{1-p-j}\ne0.
\]
Thus, for $b_j=a_j-\tilde a_j$, we have $\mathcal T b_j=\kappa_j\sigma_\eps^{-1}a_0(m_j-\tilde m_j)$. Since $\mathcal T a_0=0$, division by $a_0$ gives $\frac{\d}{\d s}(b_j/a_0)=\kappa_j\sigma_\eps^{-1}(m_j-\tilde m_j)$ along each characteristic. The incoming value of $b_j$ is zero and its outgoing value is $O(h)$. Letting $h\to0+$ therefore gives the ordinary Euclidean X-ray transform of $m_j-\tilde m_j$, and its injectivity yields $m_j=\tilde m_j$. The conclusion follows by induction.  
\end{proof} 

\section{Further remarks}
\subsection{Physical considerations} \label{sec_phys}
 
In the ultrasound literature, attenuation is commonly parametrized by a power law
\[
\alpha(f)=\alpha_0 f^y ,
\]
where the frequency $f$ is measured in MHz, $\alpha(f)$ is the attenuation coefficient in $\mathrm{dB}/\mathrm{cm}$, and $\alpha_0$ has units $\mathrm{dB}\,\mathrm{cm}^{-1}\,\mathrm{MHz}^{-y}$. The leading high-frequency attenuation rate behaves as $f^{1-p}$,  so that the experimental exponent $y$ corresponds to $y=1-p$. Thus the  linear attenuation law $y=1$ corresponds to the endpoint $p=0$, while intermediate values $0<y<1$ give genuinely fractional exponents $0<p<1$. Standard medical-ultrasound estimates place soft-tissue attenuation at about $0.3$--$0.8\,\mathrm{dB}/\mathrm{cm}$ at $1\,\mathrm{MHz}$, often rounded to the rule of thumb $\alpha(f)\simeq0.5 f_{\mathrm{MHz}}\,\mathrm{dB}/\mathrm{cm}$ \cite{Bushberg2012}. 
The law is then approximately linear, and the half-amplitude distance is
\[
d_{1/2}(f)=\frac{6.02}{0.5 f_{\mathrm{MHz}}}\,\mathrm{cm}
\simeq \frac{12}{f_{\mathrm{MHz}}}\,\mathrm{cm}.
\]
Thus $d_{1/2}\simeq 12\,\mathrm{cm}$ at $1\,\mathrm{MHz}$, $2.4\,\mathrm{cm}$ at $5\,\mathrm{MHz}$, and $1.2\,\mathrm{cm}$ at $10\,\mathrm{MHz}$. 
%
The AIUM \cite{AIUM2019} also uses homogeneous tissue models with attenuation $0.3\,\mathrm{dB}/(\mathrm{cm}\,\mathrm{MHz})$, and notes that other power laws, for example $0.75 f_{\mathrm{MHz}}^{1.5}\,\mathrm{dB}/\mathrm{cm}$ (which is outside our  range), are used in breast-imaging models. These values show that MHz ultrasound gives attenuation distances on macroscopic, clinically relevant scales. Intermediate exponents also occur experimentally: Chintada, Rau, and Goksel \cite{ChintadaRauGoksel2022}  fit the model $\alpha(f)=\alpha_0 f^y$ in ex-vivo bovine muscle and report, at $f_c=5.2\,\mathrm{MHz}$, $\alpha(f_c)=2.95\pm0.40\,\mathrm{dB}/\mathrm{cm}$, $y=0.54\pm0.29$, and, in a second configuration, $\alpha(f_c)=4.69\pm0.89\,\mathrm{dB}/\mathrm{cm}$, $y=0.76\pm0.31$. In our notation these correspond to $p\simeq0.46$ and $p\simeq0.24$, respectively. The associated half-amplitude distance is
\[
d_{1/2}(f)=\frac{20\log_{10}2}{\alpha(f)}\simeq \frac{6.02}{\alpha(f)}
\]
when $\alpha(f)$ is measured in $\mathrm{dB}/\mathrm{cm}$, and therefore the above bovine-muscle values give $d_{1/2}(f_c)\simeq2.0\,\mathrm{cm}$ and $d_{1/2}(f_c)\simeq1.3\,\mathrm{cm}$. Thus propagation over a distance comparable to the half-amplitude distance is a physically realistic regime in which the wave amplitude is reduced by an order-one factor, but is neither essentially unchanged nor completely damped.
 
In the bovine-muscle example at $f_c=5.2\,\mathrm{MHz}$, the wavelength is approximately $0.30\,\mathrm{mm}$, while the half-amplitude distance is about $1$--$2\,\mathrm{cm}$. 
Thus one half-amplitude distance contains roughly $30$--$70$ wavelengths.
Equivalently, after rescaling distances by the half-amplitude distance, the
dimensionless wavelength is about $(1.5\text{--}3)\times10^{-2}$. With the
normalization $2\pi h/\lambda$ for the wavelength, taking $\lambda\approx1$,
this corresponds to $h\approx2$--$5\times10^{-3}$. 
This is the dimensionless interpretation of $h$ used in the geometric optics construction: up to the fixed normalization $2\pi/\lambda$, it is the wavelength measured relative to the half-amplitude propagation distance.

Seismology provides another natural class of examples, although the
standard notation there is different.  Attenuation is often described by
the quality factor $Q$, or $q=Q^{-1}$.  For a monochromatic wave of
frequency $f$ and speed $c$, the amplitude decays approximately as
\[
A(L,f)=A(0,f)\exp\left(-\frac{\pi f L}{cQ(f)}\right),
\]
so that $d_{1/2}=cQ(f)\log 2/(\pi f)$.  For instance, with
$c=3\,\mathrm{km/s}$, $Q=50$, and $f=10\,\mathrm{Hz}$, one obtains
$d_{1/2}\simeq3.3\,\mathrm{km}$.  The commonly used constant-$Q$
model corresponds to attenuation proportional to $f$, hence to the
endpoint $p=0$ in our notation; see, for example, \cite{Kjartansson1979}.
However, frequency-dependent mantle attenuation is often written in the
form $q(\omega)\sim\omega^{-\alpha}$; then the wave attenuation
coefficient is proportional to $\omega^{1-\alpha}$, corresponding to
$p=\alpha$.  Reported values such as $\alpha\simeq0.1$--$0.3$
therefore give nonzero fractional exponents in the range 
$p\simeq0.1$--$0.3$ \cite{LekicMatasPanningRomanowicz2009}.

Other viscoelastic materials obeying power laws are amorphous polymers and elastomers, see, e.g., \cite{Mainardi2010}.

\subsection{Rescaling the variables} \label{sec_rescale}
Set $\tilde t=at$, $\tilde x=bx$, and $\tilde s=as$, and define $\tilde u(\tilde t,\tilde x)=u(\tilde t/a,\tilde x/b)$. Then equation~\ref{1} preserves its form with coefficients
\[
\tilde\sigma(\tilde x)
=
\frac{b^2}{a^2}
\sigma\left(\frac{\tilde x}{b}\right), \quad 
\tilde{\mathfrak m}(\tilde s,\tilde x)
=
\frac{b^2}{a^3}
\mathfrak m\left(\frac{\tilde s}{a},\frac{\tilde x}{b}\right).
\]
Then
\[
\tilde{\mathfrak m}(\tilde s,\tilde x)
=
\tilde s^{p-1}
\tilde m(\tilde s,\tilde x),
\qquad \text{with\quad }
\tilde m(\tilde s,\tilde x)
=
\frac{b^2}{a^{p+2}}
m\!\left(\frac{\tilde s}{a},\frac{\tilde x}{b}\right).
\]

We have already used a rescaling implicitly to normalize the constant
background speed to one.  If $\tilde t=at$ and
$\tilde x=bx$, then $\tilde\sigma=\frac{b^2}{a^2}\sigma$, 
and the normalization $\tilde\sigma=1$ requires $a^2=b^2\sigma$. The original phase would be $\lambda(-t+x\cdot\omega/\sqrt{\sigma})/h$, with wavelength $2\pi h\sqrt{\sigma}/\lambda$. Under this normalization, the phase in the new variables becomes $\lambda(-\tilde t+\tilde x\cdot\omega)/(ah)$. We set $\tilde h=ah$. The wavelength in the tilded variables is then $2\pi\tilde h/\lambda$.

Up to an inessential constant, the half-amplitude distance in the rescaled variables is
\[
\tilde d_{1/2}
\sim
\frac{(ah)^{1-p}}{[\tilde m]} = 
\frac{\tilde h^{1-p}}{[\tilde m]},
\]
where $[\tilde m]$ denotes a typical size of $\tilde m(0,\cdot)$. 
 Choosing the remaining common scale so that the half-amplitude distance is of order one gives
\[
[\tilde m]\sim \tilde h^{1-p}.
\]
This relation describes the size of the coefficient in the new units; it does not require the material coefficient in the original units to depend on the wavelength.

To make the change of units explicit, keep the original material fixed and let $\ell$ denote the wavelength parameter in the original units. Once we have achieved $\sigma\sim1$, with $[m]\sim1$, the half-amplitude propagation time is of order $\ell^{1-p}$, so we consider the interval $[0,\ell^{1-p}T]$ with $T>0$ fixed.
We rescale both $t$ and $x$ to preserve $\sigma\sim1$ and make this interval $[0,T]$ in the new variables. Thus $\tilde t=\ell^{p-1}t$ and $\tilde x=\ell^{p-1}x$, i.e., $a=b=\ell^{p-1}$. The semiclassical parameter in the new variables is therefore $h=a\ell=\ell^p$, while the transformed kernel is
\[
\tilde{\mathfrak m}(\tilde s,\tilde x)
=\ell^{1-p}\mathfrak m(\ell^{1-p}\tilde s,\ell^{1-p}\tilde x)
=h^{1-p}\tilde s^{p-1}m(\ell^{1-p}\tilde s,\ell^{1-p}\tilde x).
\]
This change of variables measures distance in units of the half-amplitude distance and time in units of the corresponding travel time. It preserves the wave speed and makes the propagation scale of order one. Since it also rescales the spatial and temporal profile of the kernel, it is a change of units rather than merely a weakening of the memory in fixed coordinates.
After absorbing the rescaled arguments into a redefined smooth profile and removing the tildes, this is precisely the dimensionless form with coefficient $h^{1-p}$ used above.

\subsection{Variable $\sigma$.} \label{sec_var}
We could have worked with a nonconstant $\sigma=\sigma(x)$ without additional difficulty. Indeed, even when $\sigma=1$, the geometric optics construction already uses the variable speed $\sqrt{1+\eps\sigma_\mathfrak{m}}$; one would instead use $\sqrt{\sigma+\eps\sigma_\mathfrak{m}}$. The inverse problem remains solvable as long as $\sigma^{-1}\d x^2$ is a simple metric in $\overline{B(0,R)}$. 
The choice $\sigma=1$, however, simplifies the exposition. Allowing variable $\sigma$ also covers the equivalent formulation in which $\nabla u(t,x)$ does not appear inside the memory integral, since its contribution can be absorbed into the instantaneous coefficient by replacing $\sigma$ with $\sigma+\eps\sigma_{\mathfrak m}$. For the unscaled model, this corresponds formally to $\eps=1$. 

\subsection{Initial conditions in the memory region.} The problem we solved is essentially a near-field scattering one: probing a compactly supported (in the spatial variable) memory $\mathfrak{m}$ with waves coming from outside the support. One may want to understand wave propagation when the initial disturbance is within the $x$-support of the memory. In some special cases, such as when $\mathfrak{m}$ is independent of $x$ and nonsingular in $s$, it is known that there are also standing waves, i.e., waves with zero phase velocity. The analysis of this phenomenon in our setting requires additional effort. The difficulty is most apparent when a disturbance is introduced abruptly at $t=0$ inside the spatial support of $\mathfrak m$, while the prehistory is kept zero. For $0<p<1$, this mismatch produces fractional powers in the formal expansion near $t=0$, so the integer-order regularity and compatibility conditions of Theorem~\ref{thm_well_posed} generally fail. Internal initial data accompanied by a suitably matched prehistory may nevertheless satisfy those conditions. When $p=1$, this particular fractional-power obstruction is absent.

\section{Well-posedness}\label{sec_wp}
In this section, we prove that a nonhomogeneous version of \r{1}, see \r{meq3}, is solvable in appropriate spaces, and that sufficiently smooth, compatible initial conditions $(u_0,v_0,u_-)$ and a source $f$ lead to a smooth solution depending continuously on them at increasing levels of regularity. The main steps are the following. 

\begin{itemize}
\item We formulate the problem as a $3\times 3$ system $\dot z=Lz+g$ for $z=(u,u_t,\eta)$; see \r{meq2a}--\r{meq2d} below, initially with $g=0$. We follow mostly \cite{Chep_Pata}, which is based on ideas introduced by Dafermos \cite{Dafermos_70, Dafermos_76}. The memory term in \cite{Chep_Pata}, however, is regular and $x$-independent. Existence of a weak solution then follows from semigroup theory; see Theorem~\ref{thm_wp1}. The equation can also be placed in the general framework of hyperbolic Volterra equations of variational type; see, for example, \cite[Chapter~II, Section~6.7]{Pruss1993}. This theory provides an alternative route to energy well-posedness under related positivity and coercivity assumptions. We use the history-space formulation because it retains the prescribed prehistory as part of the state and directly yields a contraction estimate that is of independent interest.

\item We prove in Theorem~\ref{thm_well_posed} that those weak solutions are strong when the data $(u_0,v_0,u_-,g)$, $g=(0,f,0)$, are sufficiently regular and compatible. This is a nontrivial step, unlike what happens when $\mathfrak{m}=0$. The reason is that $D(L^k)$ is not directly related to the standard Sobolev spaces. We solve an elliptic Volterra problem in Lemma~\ref{lemma_V} to make the connection.

\item We derive energy-like estimates for the higher-order derivatives of $u$ in terms of the data $(u_0,v_0,u_-,f)$, uniformly for $\eps\in[0,\eps_0]$; see Corollary~\ref{cor_est}. 

\item Finally, we apply Theorem~\ref{thm_well_posed} and Corollary~\ref{cor_est} to justify the asymptotic expansion derived earlier. Proposition~\ref{pr_u_N} gives a unique smooth exact solution and shows that it is close to the asymptotic solution together with all its derivatives. 
\end{itemize}

We list our assumptions on $\mathfrak{m}(s,x)$ in addition to \r{m0}. We consider it defined in $\R_s\times \R^n_x$ with a compact $x$-support, as we postulated in \r{supp}. 
First, a natural assumption in viscoelastic theory is
\begin{equation}   \label{m-monotone}
\frac{\d}{\d s} \mathfrak{m}(s,x)\le0.
\end{equation}
This is not needed for the construction of the asymptotic solution, probably not needed for well-posedness but it is required below to show that the solution is given by a contraction semigroup. It is a physical condition; in fact, many papers assume that $\mathfrak{m}$ is the Laplace transform of a measure supported in $[0,\infty)$. 

Next, we initially assume
\be{m-int}
 \mathfrak m \in L^1((0,\infty);L^\infty_x),
\ee
but for stronger solutions, we will assume eventually that for $k\geq0$ depending on the order, 
\be{m-int-k}
\partial_x^\beta\mathfrak m
\in L^1((0,\infty);L^\infty_x),
\qquad |\beta|\leq k.
\ee
In the geometric optics part, we assume \r{m0}, which implies that the singularity at $s=0$ is integrable, while \r{m-int} gives us some control as $s\to\infty$, uniformly in $x$. 

In what follows, until section~\ref{sec_just}, we assume \r{m0}, \r{m-monotone} and \r{m-int}.

\subsection{The semigroup approach.} 
We start with proving that \r{1} has a weak solution using the semigroup theory. Here, we drop $\eps$; in other words, we incorporate it in $\mathfrak{m}$.

\subsubsection{Converting to a system}
Notice first that when the initial conditions $(u,u_t)$ at $t=0$ are compactly supported, say on $B(0,R')$ with $R'>R$, then $\supp_x u \subset B(0, R'+ct)$, where $c>0$ is such that $c^2\ge \sigma$. This follows by considering the memory integral as a source, which is supported in $B(0,R)$, and applying Duhamel's principle. Since we are interested in solutions propagating for $t\in [0,T]$ with $T>0$ fixed, we can always take a bounded domain $\Omega\subset\R^n$, large enough, so that such solutions would never hit $\p\Omega$. This allows us to study the well-posedness for $x\in \Omega$, with some convenient boundary conditions, and we choose the Dirichlet ones.

Consider \r{1}. Assume for now that all functions involved are regular enough. 
Recalling $\eta^t$ in \r{eta}, clearly
\[
(\partial_t+\partial_s) \eta^t(s) = u_t(t),\quad  t\ge0, \; s\ge0.
\]
We claim that problem \r{1} is equivalent to 
\begin{subequations}
\begin{align}
&\hspace{2em}
\partial_t^2 u(t)
-\Delta_\sigma u(t)
-\nabla \cdot \int_0^\infty
\mathfrak{m}(s,x)\nabla \eta^t(s)\,\d s
=0,\quad t>0,\; x\in\Omega, \label{meq2a} \noeqref{meq2a}
\\
&\hspace{2em}
(\partial_t+\partial_s)\eta^t(s)
=u_t(t), \quad t>0,\; s>0, \label{meq2b} \noeqref{meq2b}
\\
&\hspace{2em}
\eta^t(0)=0,\quad t>0,
\quad\text{(boundary condition at $s=0$)}, \label{meq2c} \noeqref{meq2c}
\\
&\hspace{2em}
u(0)=u_0,\quad u_t(0)=v_0,\quad
\eta^0(s)=\eta_0(s),\quad s\ge0,
\quad\text{(initial conditions at $t=0$)}. \label{meq2d} \noeqref{meq2d}
\end{align}
\end{subequations}
Here $\eta_0(s)=u_0-u_-(-s)$ and $\Delta_\sigma\coloneq\nabla\cdot\sigma\nabla$; see also Figure~\ref{fig1}. There are also implicit Dirichlet boundary conditions for $x\in\p\Omega$, imposed by the requirement that we solve \r{meq2a} in the energy space $\mathcal H$ defined below.

\begin{figure}[htbp]
\includegraphics{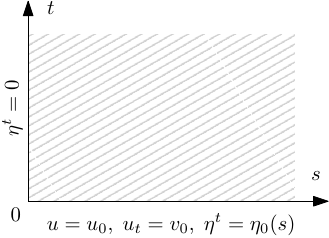}
\caption{We solve \r{meq2b}--\r{meq2d} in the product of the first quadrant with $L^2(\Omega)$. Smooth solutions require initial conditions satisfying compatibility conditions at the corner $s=t=0$.} 
\label{fig1}
\end{figure}

Taking $\eta^t$ as in \r{eta}, equation \r{1} implies \r{meq2a}--\r{meq2d}. 
Conversely, \r{meq2b} has a unique solution with boundary values on the axes of the first quadrant in the $(s,t)$-plane, and $\eta^t(s)$ in \r{eta} is clearly a solution provided that $u(t)$ is continuous everywhere, including across $t=0$, i.e., $u_-(0-)=u(0+)=u_0$. Assuming continuity away from $t=0$, the condition guaranteeing continuity there is $\eta^0(0)=0$, which is the zeroth-order compatibility condition. When this continuity fails, \r{meq2b} is still solvable, but its solution is discontinuous across the diagonal $t=s$.

We introduce $H_D(\Omega)$ as the completion of $C_0^\infty(\Omega)$ with respect to the Dirichlet norm
\[
\|u\|_{H_D(\Omega)}^2= \int_{\Omega}\sigma |\nabla u|^2\,\d x .
\]
By the Poincar\'e inequality, $H_D(\Omega)$ is topologically equivalent to $H_0^1(\Omega)$. We define the sesquilinear form
\begin{equation}\label{DefMInnProd}
(\eta_1,\eta_2)_{\mathcal M}
=
\int_0^\infty\int_\Omega  \mathfrak{m}(s,x)\,\nabla \eta_1(s,x)\cdot \overline{\nabla \eta_2(s,x)}\, \d x\,\d s.
\end{equation}
Since $\mathfrak{m}(s,x)$ may vanish in some regions, the induced quantity
$\|\eta\|_{\mathcal M}:=(\eta,\eta)_{\mathcal M}^{1/2}$ is in general only a seminorm.
Let
\[
\widetilde{\mathcal M}_0:= C_0^\infty((0,\infty)\times\Omega)
\]
be equipped with this seminorm, and let
\[
\mathcal N:=\{\eta\in \widetilde{\mathcal M}_0 : \|\eta\|_{\mathcal M}=0\}.
\]
We define $\mathcal M$ as the completion of the quotient
$\widetilde{\mathcal M}_0/\mathcal N$ with respect to the induced norm.
Then $\mathcal H = H_D(\Omega)\oplus L^2(\Omega)\oplus \mathcal M$ is the natural energy space for
$(u(t,\cdot),u_t(t,\cdot),\eta^t)$.

For $t\ge 0$, define the translation operator on representatives by
\begin{equation}\label{def-shift}
(\widetilde S(t)\eta)(s,x):=
\begin{cases}
0, & 0<s\le t,\\
\eta(s-t,x), & s>t.
\end{cases}
\end{equation}

\begin{lemma}[Right-shift semigroup]\label{lem:shift-semigroup}
The family $\{S(t)\}_{t\ge0}$ induced by \eqref{def-shift} is a strongly
continuous semigroup of contractions on $\mathcal M$. Let $A$ be its
infinitesimal generator. Then $A$ is closed and densely defined. Moreover,
\[
\mathcal C:=C_0^\infty((0,\infty)\times\Omega)/\mathcal N
\]
is a core for $A$, and
\[
A[\eta]=[-\partial_s\eta],
\qquad [\eta]\in\mathcal C .
\]
\end{lemma}

\begin{proof}
We divide the proof into several steps.

\emph{Step 1: $S(t)$ is well defined on the quotient and is contractive.}
For $\eta\in \widetilde{\mathcal M}_0$, using \eqref{def-shift} and the change of variables
$r=s-t$, we get
\begin{align*}
\|\widetilde S(t)\eta\|_{\mathcal M}^2
&=
\int_0^\infty\int_\Omega \mathfrak{m}(s,x)\,|\nabla(\widetilde S(t)\eta)(s,x)|^2\,\d x\,\d s\\
&=
\int_t^\infty\int_\Omega \mathfrak{m}(s,x)\,|\nabla\eta(s-t,x)|^2\,\d x\,\d s\\
&=
\int_0^\infty\int_\Omega \mathfrak{m}(r+t,x)\,|\nabla\eta(r,x)|^2\,\d x\,\d r\\
&\le
\int_0^\infty\int_\Omega \mathfrak{m}(r,x)\,|\nabla\eta(r,x)|^2\,\d x\,\d r
=
\|\eta\|_{\mathcal M}^2,
\end{align*}
where we used the monotonicity of $\mathfrak{m}$ with respect to $s$. Hence $\widetilde S(t)$ preserves the null space
$\mathcal N$, so it descends to a well-defined contraction on
$\widetilde{\mathcal M}_0/\mathcal N$, still denoted by $S(t)$:
\[
S(t)[\eta]:=[\widetilde S(t)\eta].
\]
Since $\widetilde{\mathcal M}_0/\mathcal N$ is dense in $\mathcal M$, $S(t)$ extends continuously  to   $S(t) :\mathcal M\rightarrow\mathcal M$.

\emph{Step 2: semigroup property.}
For representatives,
\[
\widetilde S(t)\widetilde S(\tau)\eta=\widetilde S(t+\tau)\eta,
\qquad t,\tau\ge0,
\]
which is immediate from \eqref{def-shift}. Passing to equivalence classes gives
\[
S(t)S(\tau)=S(t+\tau),\qquad S(0)=I.
\]

\emph{Step 3: strong continuity on a dense subspace.}
Let $\eta\in \widetilde{\mathcal M}_0$. Since $\eta$ has compact support in
$(0,\infty)\times\Omega$, there exist $\delta>0$, $M>\delta$, and a compact set
$K\Subset \Omega$ such that
\[
\supp \eta\subset [\delta,M]\times K.
\]
For $0<t<1$, both $\eta$ and $\widetilde S(t)\eta$ are supported in
$[\delta,M+t]\times K$. Since $m$ is locally bounded, there exists $C>0$ such that
\[
\mathfrak{m}(s,x)\le C,\qquad (s,x)\in [\delta,M+1]\times K.
\]
Hence
\[
\|\widetilde S(t)\eta-\eta\|_{\mathcal M}^2
\le
C\int_0^\infty\int_\Omega
|\nabla(\widetilde S(t)\eta-\eta)(s,x)|^2\,\d x\,\d s.
\]
Now $\eta$ vanishes near $s=0$, so for small $t$ the operator $\widetilde S(t)$ is just translation
in the $s$-variable on the support of $\eta$. Therefore
\[
\nabla(\widetilde S(t)\eta)(s,x)\to \nabla \eta(s,x)
\qquad\text{pointwise as }t\downarrow0,
\]
and the convergence holds in $L^2$ on the compact set $[\delta/2,M+1]\times K$ by the standard continuity of translations. Thus
\[
\|\widetilde S(t)\eta-\eta\|_{\mathcal M}\to 0
\qquad\text{as }t\downarrow0.
\]
Equivalently,
\[
\|S(t)[\eta]-[\eta]\|_{\mathcal M}\to 0
\qquad\text{as }t\downarrow0
\]
for every $[\eta]\in \widetilde{\mathcal M}_0/\mathcal N$.

\emph{Step 4: strong continuity on $\mathcal M$.}
Since $\widetilde{\mathcal M}_0/\mathcal N$ is dense in $\mathcal M$ and each $S(t)$ is a contraction,
for any $\zeta\in\mathcal M$ and any $\zeta_n\in \widetilde{\mathcal M}_0/\mathcal N$ with
$\zeta_n\to\zeta$ in $\mathcal M$, we have
\[
\|S(t)\zeta-\zeta\|_{\mathcal M}
\le
\|S(t)(\zeta-\zeta_n)\|_{\mathcal M}
+\|S(t)\zeta_n-\zeta_n\|_{\mathcal M}
+\|\zeta_n-\zeta\|_{\mathcal M}
\le
2\|\zeta-\zeta_n\|_{\mathcal M}
+\|S(t)\zeta_n-\zeta_n\|_{\mathcal M}.
\]
First choose $n$ large and then let $t\downarrow0$. This proves
\[
\lim_{t\downarrow0}\|S(t)\zeta-\zeta\|_{\mathcal M}=0,
\qquad \zeta\in\mathcal M.
\]
Hence $\{S(t)\}_{t\ge0}$ is a strongly continuous semigroup of contractions on $\mathcal M$.

\emph{Step 5: action of the generator on smooth elements.}
Let $A$ be the generator of $\{S(t)\}_{t\ge0}$. Fix $\eta\in \widetilde{\mathcal M}_0$. Since
$\eta$ vanishes near $s=0$, for all sufficiently small $t>0$ we have
\[
\frac{\widetilde S(t)\eta-\eta}{t}(s,x)
=
\frac{\eta(s-t,x)-\eta(s,x)}{t} \to -\partial_s\eta(s,x)
\qquad\text{as }t\downarrow0.
\]
The same localization as in Step 3, together with differentiability of translations in $L^2$ on $C_0^\infty$ functions, shows that the convergence holds in $\mathcal M$. Therefore
\[
\lim_{t\downarrow0}\frac{S(t)[\eta]-[\eta]}{t}=[-\partial_s\eta]
\quad\text{in }\mathcal M.
\]
Hence $[\eta]\in D(A)$ and
\[
A[\eta]=[-\partial_s\eta].
\]

\emph{Step 6: the core property.}
It remains to show that
\[
\mathcal C:=\widetilde{\mathcal M}_0/\mathcal N
=
C_0^\infty((0,\infty)\times\Omega)/\mathcal N
\]
is a core for $A$. By construction of $\mathcal M$, the space
$\mathcal C$ is dense in $\mathcal M$. 
By Step 5,
\[
\mathcal C\subset D(A),
\qquad
A[\eta]=[-\partial_s\eta],
\qquad [\eta]\in\mathcal C .
\]
Therefore the standard invariant-core theorem for $C_0$-semigroups, see \cite[Theorem~X.49]{Reed-Simon2}, implies
that $\mathcal C$ is a core for $A$. This proves the final assertion of the
lemma.
\end{proof}

For $v\in C_0^\infty(\Omega)$, choose
$\chi_j\in C_0^\infty((0,\infty))$, $0\leq\chi_j\leq1$, such that
$\chi_j(s)\to1$ for every $s>0$. Then
\[
\eta_j(s,x):=[\chi_j(s)v(x)]
\]
is a Cauchy sequence in $\mathcal M$, and its limit is independent of
the particular choice of $\chi_j$. We denote this limit by $Jv$.
Moreover,
\[
\|Jv\|_{\mathcal M}^2
=
\int_\Omega \sigma_{\mathfrak m}(x)|\nabla v(x)|^2\,\d x
\leq C\|v\|_{H_D}^2.
\]
Hence $J$ extends uniquely to a bounded operator
\[
J:H_D(\Omega)\longrightarrow\mathcal M.
\]
Formally, $Jv$ is the $s$-independent history $Jv(s,x)=v(x)$.

For sufficiently regular solutions, the system
\eqref{meq2a}--\eqref{meq2d} can be written as the abstract evolution
equation
\be{L}
\frac{\d}{\d t} z = Lz, \quad z(0)=z_0,
\ee
where $z(t) = (u(t),v(t),\eta^t)^T$,   $z_0 = (u_0,v_0,\eta_0)^T$, and 
\be{LDef}
L (u,v,\eta)^T = \Big(v, \Delta_\sigma u+ \nabla \cdot \int_0^\infty \mathfrak{m}(s,\cdot)\nabla \eta(s)\,\d s, A\eta+ Jv\Big)^T
\ee
with the domain 
\be{domain_L}
D(L) = \left\{ z\in \mathcal H|\;  v\in H_D, \, \Delta_\sigma u+ \nabla \cdot \int_0^\infty \mathfrak{m}(s,\cdot)\nabla \eta(s)\,\d s\in L^2, \; \eta\in D(A) \right\}.
\ee
The boundary condition \eqref{meq2c} is encoded in the choice of $A$:
$A$ is the generator of the zero-filled right-shift semigroup. Thus
$D(A)$ incorporates the zero-inflow condition at $s=0$; on the
smooth core $\mathcal C$ this is simply the condition
$\eta(0)=0$.

The domain $D(L)$ is dense in $\mathcal H$. Indeed,
$C_0^\infty(\Omega)\oplus C_0^\infty(\Omega)\oplus\mathcal C\subset D(L)$,
and this subspace is dense in
\[
H_D(\Omega)\oplus L^2(\Omega)\oplus\mathcal M .
\]
\begin{remark} 
The operator
\[
Q\eta
=
\nabla\cdot\int_0^\infty \mathfrak m(s,x)\nabla\eta(s,x)\,ds
\]
is naturally interpreted as an element of $H_D(\Omega)'$. Namely, for $\eta\in C_0^\infty((0,\infty)\times\Omega)$ and
$\varphi\in H_D(\Omega)$, set
\[
\langle Q\eta,\varphi\rangle
:=
-\int_0^\infty\int_\Omega
\mathfrak m(s,x)\nabla\eta(s,x)\cdot\overline{\nabla\varphi(x)}
\,dx\,ds .
\]
This definition passes to the quotient defining $\mathcal M$. Indeed, if
$\|\eta\|_{\mathcal M}=0$, then
$\mathfrak m^{1/2}\nabla\eta=0$ a.e., and hence the above pairing vanishes for
every $\varphi\in H_D(\Omega)$.

Moreover, since $\sigma_{\mathfrak{m}}(x)\le C\sigma(x)$, 
\[
|\langle Q\eta,\varphi\rangle|
\le
C\|\eta\|_{\mathcal M}\|\varphi\|_{H_D}.
\]
Consequently,  $Q$ extends to the continuous map
\be{Q}
Q:\mathcal M\longrightarrow H_D(\Omega)'.
\ee
 In general, however, $Q\eta$ need not belong to $L^2(\Omega)$. 
We similarly define the operator $u\in H_D(\Omega)\mapsto\Delta_\sigma u\in H_D(\Omega)'$ by $\langle \Delta_\sigma u,\varphi\rangle=-\langle\sigma\nabla u,\nabla\varphi\rangle$ for all $\varphi\in H_D(\Omega)$.
With this the second component of \r{LDef} makes sense. 
\end{remark}

We will show that $L$ is dissipative.  
For $z\in D(L)$,
\[
(Lz,z)_{\mathcal H}=(v,u)_{H_D} +\langle \Delta_\sigma u+Q\eta, v\rangle
+(A\eta+Jv,\eta)_{\mathcal M}.
\]
Since $\Delta_\sigma u+Q\eta\in L^2(\Omega)$, the second term is initially an
$L^2$-pairing. We then split it in $H_D(\Omega)'$ -- $H_D(\Omega)$ duality.
The real parts of $(v,u)_{H_D}$ and $\langle\Delta_\sigma u,v\rangle$ 
cancel, and the real parts of $\langle Q\eta,v\rangle$ and 
$(Jv,\eta)_{\mathcal M}$
cancel as well. 
Hence  
\be{LDissip}
\Re(Lz,z)_{\mathcal H}
=
\Re(A\eta,\eta)_{\mathcal M}
\le0
\ee
because $A$ is the generator of a contraction semigroup. Thus $L$ is dissipative.

It remains to show that $\Id-L$ is surjective.

\subsubsection{Surjectivity of $\Id-L$}

Here we follow Section 3 of \cite{Giorgi_P_01}.
Letting $L$ be the operator \r{LDef}, we desire to show that $\Id-L$ is surjective.
Given $z^\ast= \left(u^\ast,v^\ast,\eta^\ast \right)\in \mathcal H$, we want to solve
\[ 
\left(\Id-L \right) z=z^\ast
\]
for $z\in D(L)$, 
which is explicitly given by the system:
\begin{equation}\label{I_Minus_L_Syst}
	\begin{cases}
		u - v = u^\ast,\\[4pt]
		v - \Delta_\sigma u - \nabla\cdot\displaystyle\int_{0}^{\infty}
		\mathfrak{m}(s,\cdot)\,\nabla \eta(s)\,ds = v^\ast,\\[6pt]
		\eta -A \eta - Jv = \eta^\ast.
	\end{cases}
\end{equation} 
 The third equation in \r{I_Minus_L_Syst} is equivalent to
\[
(\Id-A)\eta=\eta^\ast+Jv.
\]
Since $A$ is the generator of a contraction semigroup on $\mathcal M$, we
have that $1$ is in  the resolvent set of   $A$ by  the Hille-Yosida Theorem. 
Hence, once $v\in H_D(\Omega)$ is chosen, we define
\[
\eta=(\Id-A)^{-1}(\eta^\ast+Jv).
\]
 Then automatically
\[
\eta\in D(A),
\qquad
\eta-A\eta=\eta^\ast+Jv.
\]
Thus the third equation in \r{I_Minus_L_Syst} is satisfied.

Write $R=(\Id-A)^{-1}$. 
For any $g\in C^\infty(\mathbb{R}) : \sup_{[0,\infty)}{|g|}<\infty$, multiplication by $g$ defines a bounded operator $M_g:\mathcal M\to\mathcal M$. 
Indeed, on
$\widetilde{\mathcal M}_0/\mathcal N$ we set
\[
M_g[\xi]=[g\xi].
\]
This is well defined and satisfies
$\|M_g[\xi]\|_{\mathcal M}\leq
\|g\|_{L^\infty [0,\infty)}\|[\xi]\|_{\mathcal M}$; it therefore extends uniquely
to the completion $\mathcal M$. 

 Set $g=(1-e^{-s})$ and define the bounded operator $B=M_gJ:H_D(\Omega)\to\mathcal M$. For
$v\in C_0^\infty(\Omega)$, the element $Bv$ is the limit of the smooth
histories $[g(s)\chi_j(s)v(x)]$ used in the definition of $J$. A direct
calculation of the difference quotient for the shift semigroup gives
$Bv\in D(A)$, with $ABv$  equal to 
$-e^{-s} J v(x)$. The convergence of the difference quotient in $\mathcal M$
follows from \r{m-int} by dominated convergence. Consequently,
\[
(\Id-A)Bv=Jv,
\qquad v\in C_0^\infty(\Omega),
\]
and hence $Bv=RJv$. Since both $B$ and $RJ$ are bounded from
$H_D(\Omega)$ to $\mathcal M$, density gives
$B=RJ$ on $H_D(\Omega)$. Thus the definition of $\eta$ above can be
written as
\[
\eta=R\eta^\ast+Bv.
\]

The first equation in \r{I_Minus_L_Syst} gives $u=u^\ast+v$. Substitution
in the second equation shows that $v\in H_D(\Omega)$ has to satisfy
$v-\Delta_\sigma v-QBv=v^\ast+\Delta_\sigma u^\ast+QR\eta^\ast$ in
$H_D(\Omega)'$. Set
$\widetilde\sigma(x)=\int_0^\infty
\mathfrak m(s,x)(1-e^{-s})\,\d s$. Since $Jv$ is the
$s$-independent history associated with $v$ and $B=M_gJ$, the
definitions of $M_g$, $J$, and $Q$ give
$QBv=\Delta_{\widetilde\sigma}v$, first for   $v \in C_0 ^\infty(\Omega)$ and then for
every $v\in H_D(\Omega)$ by continuity. Hence $v$ has to solve
\[
-\Delta_{\sigma+\widetilde\sigma}v+v
=v^\ast+\Delta_\sigma u^\ast+QR\eta^\ast
\quad\text{in }H_D(\Omega)'.
\]
Assumption \r{m-int} implies that
$\widetilde\sigma\in L^\infty(\Omega)$, and
$\widetilde\sigma\geq0$. Thus $\sigma+\widetilde\sigma$ is bounded
and uniformly positive. The right-hand side belongs to $H_D(\Omega)'$
because $v^\ast\in L^2(\Omega)$, $u^\ast\in H_D(\Omega)$, and both
$R:\mathcal M\to\mathcal M$ and
$Q:\mathcal M\to H_D(\Omega)'$ are bounded. The Lax--Milgram lemma,
applied to the standard weak Dirichlet formulation for uniformly elliptic
divergence-form operators, therefore gives a unique $v\in H_D(\Omega)$.

With this $v$, set $u=u^\ast+v$ and
$\eta=R(\eta^\ast+Jv)$. Then $u\in H_D(\Omega)$,
$\eta\in D(A)$, and all three equations in \r{I_Minus_L_Syst} hold.
In particular, its second equation gives
$\Delta_\sigma u+Q\eta=v-v^\ast\in L^2(\Omega)$. Hence
$z=(u,v,\eta)\in D(L)$, which proves the surjectivity of $\Id-L$.


\subsubsection{Mild and weak solutions}
By the Lumer--Phillips theorem, see, e.g., \cite[Theorem~X.48]{Reed-Simon2},
we have proved the following.
\begin{theorem} \label{thm_weak_sol} 
Assume \r{m0}, \r{m-monotone}, \r{m-int}.  
	Then the operator $L$ generates a strongly continuous contraction semigroup $e^{tL}$, $t\ge0$, on $\mathcal H$. 
\end{theorem}

In particular, 
the theorem shows that the abstract energy 
\be{AbstrEner}
E(t):=\|v(t)\|_{L^2}^2+\|u(t)\|_{H_D}^2+\|\eta^t\|_{\mathcal M}^2
\ee
is nonincreasing: $E(t)\leq E(s)$ for $0\leq s\leq t$. 
After the history variable is identified with the prescribed full history in the
proof below, this energy  will have the concrete expression
\[
E(t) = \|u_t(t)\|_{L^2}^2 + \|\sqrt{\sigma} \nabla u(t)\|_{L^2}^2 + \int_0^\infty\int_\Omega
\mathfrak m(s,x)
|\nabla(u(t,x)-u(t-s,x))|^2\,dx\,ds .
\]
%
%
%

Consider now the nonhomogeneous problem: 
\be{dz}
\frac{\d}{\d t} z= Lz+ g(t),\quad z(0)=z_0,
\ee 
where $z_0\in\mathcal H$, $g\in L^1((0,T);\mathcal{H})$ for every $T>0$. For fixed $t>0$, $e^{(t-s)L}g(s)$ is integrable; hence we consider the function given by Duhamel's formula:
\be{zt}
z(t) = e^{tL}z_0+ \int_0^t e^{(t-s)L} g(s)\, \d s
\ee
This is called the \textit{mild solution} of \r{dz}. A function $z(t)$ is called a \textit{classical solution} if
$z\in C([0,T];D(L))\cap C^1((0,T];\mathcal H)$ and it satisfies
\r{dz} for every $t\in[0,T]$.
A mild solution is not necessarily a
classical solution.

Adding a source term $f(t,x)$ to the right-hand side of \r{1} results in $g(t) = (0,f(t),0)$. This yields the following.

Put $\rho(s)=\|\mathfrak m(s,\cdot)/\sigma\|_{L^\infty}$. Since
\[
\|[F]\|_{\mathcal M}^2
\leq \int_0^\infty \rho(s)\|F(s)\|_{H_D}^2\,\d s,
\]
the class map on smooth histories extends continuously to
\[
\iota:L^2((0,\infty),\rho(s)\,\d s;H_D)\longrightarrow\mathcal M.
\]

\begin{theorem}\label{thm_wp1} 
Assume \r{m0}, \r{m-monotone}, \r{m-int}.  
Let $f\in C^0([0,T];L^2)$ for some $T>0$, let
$(u_0,v_0)\in H_D\oplus L^2$, and let $u_-$ be an
$H_D(\Omega)$-valued measurable prehistory on $(-\infty,0)$ such that
$F_0(s):=u_0-u_-(-s)$ belongs to
$L^2((0,\infty),\rho(s)\,\d s;H_D)$. Set $\eta_0:=\iota(F_0)$,
$z_0=(u_0,v_0,\eta_0)$ and $g(t)=(0,f(t),0)$. Then \r{dz} has a
unique mild solution $z\in C([0,T];\mathcal H)$, given by \r{zt}.
\end{theorem}

We connect the notion of a mild solution to \r{dz} to that of a weak solution of 
\be{meq3}
	\begin{dcases}
		\partial_t^2 u	-\Delta_\sigma u -  \nabla \cdot \int_0^\infty \mathfrak{m}(s,\cdot)\big( \nabla u(t) - \nabla u(t-s)\big)\,\d s=f(t,\cdot),\quad &t>0,\\
		u(0)=u_0, \quad u_t(0)=v_0, \quad u(t)= u_-(t)\  \text{for $t<0$}.
	\end{dcases} 
	\ee

\begin{theorem}\label{thm_wp2} 
Under the assumptions of Theorem~\ref{thm_wp1}, problem \r{meq3} 
has a unique weak solution 	$u\in C([0,T];H_D ) \cap C^1([0,T];L^2 ) \cap C^2([0,T];H_D')$,  given by the first component of the mild solution
$z\in C([0,T];\mathcal H)$ corresponding to $z_0=(u_0,v_0,\eta_0)$.
	

	Moreover, we have that the following energy estimate holds
	\begin{equation}\label{eq:energy-source-sharp}
		E(t)^{1/2} 		\le 
		E(0)^{1/2}  + \int_0^t   \|f(s)\|_{L^2}\,\d s.
	\end{equation}
\end{theorem}

 \begin{remark}
 In the theorem, a weak solution means that $u$,
$u_t$, and $u_{tt}$ have the stated time regularity, that its associated
full-history variable belongs to $C([0,T];\mathcal M)$, and that
\r{meq3} holds in $H_D'$ for every $t$, together with the initial and
prehistory conditions. Equivalently, one may use the classical
space--time formulation obtained by testing against smooth functions
$\phi(t,x)$ satisfying $\phi|_{\partial\Omega}=0$ and
$\phi(T,\cdot)=\partial_t\phi(T,\cdot)=0$; integration by parts then
incorporates the initial conditions.
\end{remark}

\begin{proof}[Proof of Theorem~\ref{thm_wp2}]
Let $z(t)$ be the mild solution given by \r{zt}, with
$g(t)=(0,f(t),0)$. By \r{zt} and contractivity of the semigroup,
\[
\|z(t)\|_{\mathcal H}
\leq \|z_0\|_{\mathcal H}
+\int_0^t\|f(s)\|_{L^2}\,\d s.
\]
Once the history variable is identified below, this is precisely
\eqref{eq:energy-source-sharp}.
It remains to identify the history component and show that the first
component of $z$ solves \r{meq3}. 
For a general mild solution, $v(t)$ need not belong to $H_D$, and
hence $Jv(t)$ need not be defined in $\mathcal M$. We therefore argue
by approximation. Choose
\[
z_{0,n}\in D(L),
\qquad
f_n\in C^1([0,T];L^2(\Omega)),
\]
such that
\[
z_{0,n}\to z_0\quad\text{in }\mathcal H,
\qquad
f_n\to f\quad\text{in }C([0,T];L^2(\Omega)).
\]
Let $z_n=(u_n,v_n,\eta_n)$ be the mild solution corresponding to
$z_{0,n}$ and $g_n=(0,f_n,0)$. Since $z_{0,n}\in D(L)$ and
$g_n\in C^1([0,T];\mathcal H)$, standard semigroup regularity implies
that $z_n$ is a classical solution. 
Duhamel's formula and contractivity
then give 
\[
z_n\longrightarrow z
\quad\text{in }C([0,T];\mathcal H).
\]
For each $n$ and $t$, we have $v_n(t)\in H_D$. Applying the variation-of-constants formula to the history equation gives
\[
\eta_n^t 
=
S(t)\eta_n^0 
+
\int_0^t S(t-r)J v_n(r) \,\d r.
\]
In particular, $Jw=\iota(s\mapsto w)$. For $u\in C([0,T];H_D)$, let
\[
G_{t,u}(s)
=
\begin{cases}
u(t)-u(t-s),&0<s<t,\\
u(t)-u(0),&s>t.
\end{cases}
\]
and define $\Gamma_tu:=\iota(G_{t,u})$. If
$u\in C^1([0,T];H_D)$, the definitions of $S(t)$ and $J$, followed by the
fundamental theorem of calculus, give
\[
\Gamma_tu=\int_0^tS(t-r)J\partial_ru(r)\,\d r.
\]
In particular, since $v_n=\partial_tu_n$, the identity above
can be written as $\eta_n^t=S(t)\eta_n^0+\Gamma_tu_n$. Moreover,
\[
\|\Gamma_tu\|_{\mathcal M}\le 2C_{\mathfrak m}^{1/2}
\|u\|_{C([0,T];H_D)}, \qquad
C_{\mathfrak m}=\int_0^\infty\left\| \mathfrak m(s,\cdot)/\sigma \right\|_{L^\infty}\,\d s,
\]
uniformly for $t\in[0,T]$.   Assumption \r{m-int}
implies that $C_{\mathfrak m}<\infty$, so $\Gamma_t$ is bounded on
$C([0,T];H_D)$, uniformly in $t$. Passing to the limit in $\mathcal M$, we
obtain
\be{eta_mild}
\eta^t=S(t)\eta^0+\Gamma_tu.
\ee
Recall that $F_0(s)=u_0-u_-(-s)$ and
$\eta^0=\eta_0=\iota(F_0)$. For $t\in[0,T]$, define
\[
F_t(s)=
\begin{cases}
u(t)-u(t-s),&0<s<t,\\
u(t)-u_-(t-s),&s>t.
\end{cases}
\]
The definitions of $\iota$, $S(t)$, and $\Gamma_t$, first for smooth
histories and then by continuity, show that identity \r{eta_mild} is
precisely
\[
\eta^t=\iota(F_t)\qquad\text{in }\mathcal M.
\]
Thus the displayed formula defines a concrete history whose image under
$\iota$ is $\eta^t$; it does not involve pointwise evaluation of an
arbitrary element of $\mathcal M$.

To interpret the memory flux, define on
$\widetilde{\mathcal M}_0/\mathcal N$
\[
P\xi
:=\int_0^\infty\mathfrak m(s,\cdot)\nabla\xi(s,\cdot)\,\d s.
\]
The Cauchy--Schwarz inequality gives
\[
\|P\xi\|_{L^2(\Omega;\mathbb R^n)}^2
\leq \|\sigma_{\mathfrak m}\|_{L^\infty}
\|[\xi]\|_{\mathcal M}^2.
\]
Hence $P$ is well defined on the quotient and extends continuously to
a map $P:\mathcal M\to L^2(\Omega;\mathbb R^n)$.
For every $F\in L^2((0,\infty),\rho(s)\,\d s;H_D)$, approximation by
smooth histories gives
\[
P\iota(F)=\int_0^\infty
\mathfrak m(s,\cdot)\nabla F(s,\cdot)\,\d s,
\]
where the integral is understood as the $L^2(\Omega;\mathbb R^n)$-limit
of the corresponding integrals for smooth approximants. Applying this
identity to $F_t$ gives the following equality in
$L^2(\Omega;\mathbb R^n)$:
\[
\begin{aligned}
P\eta^t
&=
\int_0^t
\mathfrak m(s,x)
\nabla\bigl(u(t,x)-u(t-s,x)\bigr)\,\d s
\\
&\quad+
\int_t^\infty
\mathfrak m(s,x)
\nabla\bigl(u(t,x)-u_-(t-s,x)\bigr)\,\d s.
\end{aligned}
\]
%
Since $Q=\nabla\cdot P:\mathcal M\to H_D'$, we may eliminate
$\eta^t$ from the second component of the abstract system and obtain
the original full-history equation in $\mathcal D'(0,T;H_D')$. 

The same approximation argument gives
\[
u_t=v
\quad\text{in }\mathcal D'(0,T;L^2(\Omega))
\]
and
\[
v_t=\Delta_\sigma u+ Q\eta+f
\quad\text{in }\mathcal D'(0,T;H_D').
\]
Since $v\in C([0,T];L^2(\Omega))$, the first identity implies
\[
u\in C^1([0,T];L^2(\Omega)).
\]
Moreover,
\[
\Delta_\sigma u+ Q\eta+f
\in C([0,T];H_D').
\]
Since
\[
v\in C([0,T];L^2(\Omega))
\subset C([0,T];H_D'),
\]
and its distributional derivative is continuous with values in $H_D'$,
it follows that $v\in C^1([0,T];H_D')$. Hence
\[
u_{tt}
=
\Delta_\sigma u+ Q\eta+f
\in C([0,T];H_D'),
\]
which proves the claimed regularity. The initial conditions in \r{meq3}
are satisfied by construction. Conversely, let $u$ be any weak solution
and let $\eta$ be its associated history variable. Then
$z=(u,u_t,\eta)\in C([0,T];\mathcal H)$ satisfies
\[
\frac{\d}{\d t}(z(t),\zeta)_{\mathcal H}
=(z(t),L^*\zeta)_{\mathcal H}
+(g(t),\zeta)_{\mathcal H}
\quad\text{in }\mathcal D'(0,T),
\qquad \zeta\in D(L^*),
\]
with $z(0)=z_0$. Thus $z$ is a weak solution of the abstract Cauchy
problem. By the standard uniqueness theorem for a $C_0$-semigroup, it
coincides with the mild solution \r{zt}. This proves uniqueness, and
\eqref{eq:energy-source-sharp} follows from the contractivity estimate
above.
\end{proof}

\begin{remark}\label{remark_support1}
In the wave-propagation part of the paper, we assume that the $x$-support of $\mathfrak m(s,\cdot)$ is contained in a fixed ball $B(0,R)$. In that case, the prehistory $u_-(t,x)$, $t<0$, needs to be prescribed only on $(-\infty,0)\times B(0,R)$, since for $s>t$ it enters the memory term only after multiplication by $\mathfrak m(s,x)$. Our choice of the initial conditions \r{ICa} makes the prehistory in $B(0,R)$ zero. 
\end{remark}

We have kept the memory kernel $\mathfrak m$ so far, instead of $\eps\mathfrak{m}$. We switch to $\eps\mathfrak{m}$ below. 

\subsection{Strong solutions, well-posedness}
\subsubsection{Compatibility conditions} 
We study solutions $u$ on $[0,T]\times\bar\Omega$ of
\be{full_eq}
u_{tt}
-\Delta_{\sigma_\eps} u
+\varepsilon\nabla\cdot \int_0^\infty \mathfrak m(s,x)
\nabla u(t-s,x)\,ds=f(t,x)
\ee
with $0\leq\eps\leq\eps_0$, where $\eps_0$ is fixed. We display the
dependence on $\eps$ because certain estimates below will be uniform in
$\eps$. The precise finite-order assumptions on the data, including the
forcing term $f$, are included in the definition below. For the moment, to
motivate the compatibility conditions, suppose that all quantities are
smooth enough for the following formal differentiations. Assume that a
prehistory is prescribed by
\[
u(t,x)=u_-(t,x),\qquad t<0 .
\]
The equation samples the trajectory on both sides of $t=0$. Indeed, if we
introduce the joined trajectory
\[
\widetilde u(t):=
\begin{cases}
u_-(t),&t<0,\\
u(t),&t\geq0,
\end{cases}
\]
then the memory term contains $\widetilde u(t-s)$. Splitting its integral at
$s=t$ and differentiating in $t$ produces endpoint terms involving the two
one-sided jets at $t=0$. These endpoint terms cancel precisely when the past
and forward jets agree. Thus compatibility of order $k$ requires the joined
trajectory to have matching time derivatives through order $k$ at $t=0$.

At order zero this gives
\[
u_-(0-)=u_0.
\]
For the concrete initial history
\[
\widehat\eta_0^{\,0}(s):=u_0-u_-(-s),\qquad s\geq0,
\]
this is equivalently $\widehat\eta_0^{\,0}(0)=0$. This statement concerns the
concrete history function; it does not assign a value at $s=0$ to an abstract
element of the memory space. At first order the matching condition is
\[
\partial_tu_-(0-)=v_0,
\]
or, equivalently,
$\partial_s\widehat\eta_0^{\,0}(0)=v_0$.

Beginning at second order, the forward jets are determined by the equation.
For example, evaluation of \r{full_eq} at $t=0+$ gives
\be{11}
U_2:=u_{tt}(0+)
=\Delta_{\sigma_\eps}u_0
-\varepsilon\nabla\cdot\int_0^\infty
\mathfrak m(s,\cdot)\nabla u_-(-s,\cdot)\,\mathrm ds
+f(0,\cdot).
\ee
Consequently, second-order corner compatibility is
$\partial_t^2u_-(0-)=U_2$; in terms of the concrete history above, it is
$\partial_s^2\widehat\eta_0^{\,0}(0)=-U_2$. Higher forward jets are obtained
in the same way by differentiating \r{full_eq} and evaluating at $t=0+$.
This produces the recursion below and explains
why the time jet of $f$ enters the higher-order conditions.

There is also a distinct boundary compatibility requirement. Since the
solution satisfies the homogeneous Dirichlet condition for $t\geq0$, every
forward jet that possesses a boundary trace must have zero trace. At order
$k$, this is imposed on the jets through $U_k$. The next jet $U_{k+1}$ has
only $L^2$ regularity in the scale used below, so neither a boundary trace nor
an additional matching condition is required of it. 
Based on these observations, we define the forward time jets by  
\[
U_0:=u_0 ,
\qquad
U_1:=v_0, 
\]
\begin{equation}\label{eq:forward-jets}
\begin{aligned}
U_{j+2}
:={}&
\Delta_{\sigma_\varepsilon}U_j
-
\varepsilon\nabla\cdot
\int_0^\infty
\mathfrak m(s,\cdot)\,
\nabla\partial_t^j u_-(-s,\cdot)\,\mathrm ds
+
\partial_t^j f(0,\cdot), \quad j\ge0. 
\end{aligned}
\end{equation}
Here the formula is understood whenever its right-hand side is defined. For
compatibility of order $k$, we use it for $0\leq j\leq k-1$. Under the
regularity assumptions in Definition~\ref{def:compatibility}, together with
\r{m-int-k}, its right-hand side belongs to $H^{k-1-j}(\Omega)$, and hence
$U_j\in H^{k+1-j}(\Omega)$ for $0\leq j\leq k+1$.

\begin{definition}[Compatibility conditions]\label{def:compatibility}
Fix $k\geq1$. Assume $ u_0\in H^{k+1}(\Omega)$, $v_0\in H^k(\Omega)$, 
and
\begin{equation}\label{eq:prehistory-regularity}
\partial_t^\ell u_-
\in
C_{\mathrm b}\bigl((-\infty,0];
H^{k+1-\ell}(\Omega)\bigr),
\qquad 0\leq\ell\leq k.
\end{equation}
Assume also that
\begin{equation}\label{eq:forcing-regularity}
\partial_t^\ell f
\in
C\bigl([0,T];H^{k-\ell}(\Omega)\bigr),
\qquad 0\leq\ell\leq k.
\end{equation}
In particular, $\partial_t^k f\in C([0,T];L^2(\Omega))$.

We say that $(u_0,v_0,u_-,f)$ satisfies the compatibility conditions of
order $k$ if the following two conditions hold.

(i) The forward jets agree with the corresponding prehistory jets:
\begin{equation}\label{eq:corner-compatibility}
\partial_t^j u_-(0-,\cdot)=U_j
\quad\text{in }H^{k+1-j}(\Omega),
\qquad 0\leq j\leq k.
\end{equation}

(ii) The forward jets satisfy the Dirichlet boundary condition:
\begin{equation}\label{eq:boundary-compatibility}
U_j\in H^{k+1-j}(\Omega)\cap H_0^1(\Omega),
\qquad 0\leq j\leq k.
\end{equation}
 
For $0\leq j\leq k$, define the concrete differentiated initial history
by
\begin{equation}\label{eq:concrete-differentiated-history}
\widehat\eta_j^{\,0}(s,\cdot)
:=
U_j-\partial_t^j u_-(-s,\cdot),
\qquad s\geq0.
\end{equation}
Then \eqref{eq:corner-compatibility} is equivalently
\[
\widehat\eta_j^{\,0}(0,\cdot)=0
\quad\text{in }H^{k+1-j}(\Omega),
\qquad 0\leq j\leq k.
\]
Here $\widehat\eta_j^{\,0}$ denotes the concrete history function. If
$\eta_j^0\in\mathcal M$ denotes the corresponding abstract history
element, no pointwise value $\eta_j^0(0)$ is being asserted.
%
\end{definition}

The recursion \r{eq:forward-jets} also defines $U_{k+1}\in L^2(\Omega)$.
This quantity will be the initial value of $\partial_t^{k+1}u$ and is needed
as the second component of the terminal abstract jet
$Z_k=(U_k,U_{k+1},\eta_k^0)\in\mathcal H$, introduced in the proof of
Theorem~\ref{thm_well_posed}. Although $U_{k+1}$ is determined by the
recursion, no matching or boundary condition is imposed on it.


\subsubsection{The elliptic Volterra equation. Relating $D(L^k)$ to Sobolev spaces}
For the homogeneous abstract problem, $z(0)\in D(L^k)$ implies that the
solution remains continuous with values in $D(L^k)$. In the nonhomogeneous
problem, however, smoothness of $f$ alone does not imply that
$(0,f,0)\in D(L^k)$. The appropriate conditions are instead the recursive
abstract compatibility conditions for the time jets, which will be verified
in the proof of Theorem~\ref{thm_well_posed}. What remains to be clarified is
how the resulting graph regularity relates to the classical Sobolev spaces.
The main difficulty is that $z\in D(L)$ implies $v\in H^1$, whereas the
second component of $Lz$, see \r{LDef}, gives only
$\Delta_\sigma u+Q\eta\in L^2$; the two terms cannot be separated directly
to conclude that $u\in H^2$.

The following lemma about the elliptic Volterra equation would be useful for this purpose.  

\begin{lemma}\label{lemma_V} 
For $k=0,1,\dots$, let $F\in C([0,T];H^k(\Omega))$, and
assume that \eqref{m-int-k} holds with $k+1$ in place of $k$. 
Then the Volterra equation 
\begin{equation}   \label{Volterra_eq}
\Delta_{\sigma_\eps} u (t) -  \eps\nabla \cdot \int_0^t \mathfrak{m}(s,\cdot) \nabla u(t-s)\,\d s=F(t)
\end{equation}
has a unique solution $u\in C([0,T]; \, H^{k+2}(\Omega)\cap H_0^1(\Omega))$. Moreover, the norm of the solution operator $F\mapsto u$ is uniformly bounded in $\eps\in [0,\eps_0]$, for $\eps_0\ll1$ possibly depending on $m$.  
\end{lemma}

\begin{proof}
We will formulate the problem in more general terms. Let 
$\mathcal{A}: \mathcal H_1 \to \mathcal H_2$ 
be an isomorphism, 
$B(t)\in L^1([0,T]_t; \, \mathcal{L}(\mathcal H_1, \mathcal H_2))$, 
and $F\in C([0,T]; \, \mathcal H_2)$. Then 
\[
\mathcal{A}u(t) + \int_0^t B(s)u(t-s)\, ds=F(t)
\]
has a unique solution $u\in  C([0,T]; \, \mathcal H_1)$ continuously depending on $F$. Set $v(t) = \mathcal{A} u(t)$. We are looking for a solution of
\be{eq_v}
v(t) + \int_0^t K(s) v(t-s)\, ds=F(t), \quad K(s) := B(s)\mathcal{A}^{-1} \in  L^1([0,T]_t; \, \mathcal{L}(\mathcal H_2, \mathcal H_2)).
\ee
This equation is solvable via Neumann series in $ C([0,T]; \, \mathcal H_2)$ when 
 $\|K\|_{ L^1([0,T]_t; \, \mathcal{L}(\mathcal H_2, \mathcal H_2))}<1$. 
The general case can be reduced to this one by setting $v_\alpha (t) = e^{-\alpha t}v(t)$ for some $\alpha\gg1$. Then $v_\alpha$ solves \r{eq_v} with $K$ replaced by $K_\alpha(s) =e^{-\alpha s}K(s)$, and $F_\alpha(t) =e^{-\alpha t} F(t)$. Then the convolution operator in \r{eq_v} has an operator norm bounded by 
\[
\int_0^T e^{-\alpha s} \|K(s)\|_{\mathcal{H}_2,\mathcal{H}_2}\,\d s,
\]
which converges to zero as $\alpha\to\infty$ by the dominated convergence theorem. Therefore, the operator norm can be made less than one for $\alpha>\alpha_0(\mathcal{A},B)>0$, and the Neumann series argument still applies. It is straightforward to check that the solution operator has the claimed continuity in the original norms. 

In the case of the lemma, $\mathcal H_1=H^{k+2}(\Omega)\cap H^1_0(\Omega)$, $\mathcal H_2=H^k(\Omega)$, $\mathcal{A}=\Delta_{\sigma_\eps,D}$, where the subscript $D$ denotes the Dirichlet realization, and $B(s)u=-\eps\nabla\cdot(\mathfrak{m}(s)\nabla u)$. To prove uniformity in $\eps$, write $\Delta_{\sigma_\eps,D}^{-1}=\Delta_D^{-1}(\Id+\eps\Delta_{\sigma_{\mathfrak{m}}}\Delta_D^{-1})^{-1}$. The inverse in the second factor exists on $H^k(\Omega)$ for $0\le\eps\ll1$. Note also that $K$ is $\eps$-small in this case, so no $\alpha$ is needed. 

Note that we need to assume \r{m-int-k} for $|\beta|\le k+1$ and for $t\in(0,T)$ only. 
\end{proof}
\subsubsection{Strong solutions}

In the next theorem, we need the stronger version \r{m-int-k}. 
For $0\leq j\leq k-1$, we use the notation
\be{past_source}
f_-^{(j)}(t):=\varepsilon\nabla\cdot\int_t^\infty
\mathfrak m(s,\cdot)\nabla\partial_t^j u_-(t-s,\cdot)\,\d s .
\ee
Thus $f_-^{(j)}$ denotes the contribution associated with the $j$th
prehistory jet.

\begin{theorem} \label{thm_well_posed}
Fix $k\in\{1,2,\dots\}$, $T>0$, and $0\leq\varepsilon\leq\varepsilon_0$.
Assume \r{m0}, \r{m-monotone}, and \r{m-int-k}, and suppose that
$(u_0,v_0,u_-,f)$ satisfies the regularity and compatibility conditions of
order $k$ in Definition~\ref{def:compatibility}, for this value of
$\varepsilon$. Then \r{full_eq}, with $u(0)=u_0$, $u_t(0)=v_0$, the
prescribed prehistory $u_-$, and the homogeneous Dirichlet boundary
condition, has a unique weak solution. This solution is strong and satisfies
\[
\partial_t^j u\in C\bigl([0,T];H^{k+1-j}(\Omega)\cap H_0^1(\Omega)\bigr),
\quad 0\leq j\leq k,\qquad
\partial_t^{k+1}u\in C\bigl([0,T];L^2(\Omega)\bigr).
\]
Moreover, $\partial_t^j u(0+)=U_j$ for $0\leq j\leq k+1$; hence the
forward and prehistory jets agree through order $k$. 
The solution depends continuously on the data in these spaces.
\end{theorem}

\begin{proof}
The proof consists of two steps. First, we translate the Sobolev
compatibility conditions into the corresponding abstract ones. With
$Z_j=(U_j,U_{j+1},\eta_j^0)$ and $g=(0,f,0)$, they take the form
$Z_j\in D(L_\varepsilon)$ and
$Z_{j+1}=L_\varepsilon Z_j+g^{(j)}(0)$ for $0\leq j\leq k-1$. When $f=0$,
this is precisely the condition $Z_0\in D(L_\varepsilon^k)$; for general
$f$, it is its source-adjusted version. Semigroup regularity then gives the
required time regularity in the energy and graph spaces. Second, after
differentiating the equation and splitting the memory integral into its
present and prehistory parts, we apply Lemma~\ref{lemma_V}, starting with
$\partial_t^{k-1}u$ and proceeding downward, to obtain the additional
spatial regularity.

We give the proof for $\varepsilon>0$; when $\varepsilon=0$, the memory
component is omitted and the same argument reduces to the standard wave
equation. Let $L_\varepsilon$ be the operator $L$ with $\mathfrak m$ replaced
by $\varepsilon\mathfrak m$, and set $g(t)=(0,f(t),0)$. For
$0\leq j\leq k$, let $\eta_j^0\in\mathcal M$ be the abstract history element
associated with the concrete function $\widehat\eta_j^{\,0}$ in
\r{eq:concrete-differentiated-history}, and set
$Z_j=(U_j,U_{j+1},\eta_j^0)$, $0\leq j\leq k$.

We first verify the abstract compatibility conditions. For $j\leq k-1$,
the concrete function $\widehat\eta_j^{\,0}$ vanishes at $s=0$, while it and
its $s$-derivative are bounded and continuous with values in the appropriate
spatial $H^1$ space. The difference-quotient definition of the generator of
the zero-filled right shift therefore gives $\eta_j^0\in D(A)$ and
$A\eta_j^0=-\partial_s\widehat\eta_j^{\,0}$ in $\mathcal M$. This assertion
follows by dominated convergence on $s>h$; on $0<s\leq h$, the zero filling
and $\widehat\eta_j^{\,0}(0)=0$ give a bound tending to zero because
$\int_0^h\|\mathfrak m(s,\cdot)\|_{L^\infty}\,\d s\to0$. Thus the value at
$s=0$ is used only for the concrete history function. The boundary
compatibility gives $U_{j+1}\in H_0^1(\Omega)$. Finally, the definition of
the forward jets gives
\[
A\eta_j^0+JU_{j+1}=\eta_{j+1}^0,
\qquad
\Delta_\sigma U_j+\varepsilon Q\eta_j^0=U_{j+2}-\partial_t^j f(0),
\qquad 0\leq j\leq k-1.
\]
Consequently, $Z_j\in D(L_\varepsilon)$ for $0\leq j\leq k-1$,
$Z_k\in\mathcal H$, and
$L_\varepsilon Z_j+g^{(j)}(0)=Z_{j+1}$. Notice that this is exactly why no
additional matching or boundary condition is imposed on $U_{k+1}$: it is the
second component of $Z_k$ and is required only to belong to $L^2(\Omega)$.

The forcing regularity in Definition~\ref{def:compatibility} gives
$g\in C^k([0,T];\mathcal H)$. Hence the standard regularity result for the
inhomogeneous abstract Cauchy problem,
applied successively to $z,z_t,\dots,z^{(k-1)}$, now yields
\[
z\in C^{k-1}\bigl([0,T];D(L_\varepsilon)\bigr)
\cap C^k\bigl([0,T];\mathcal H\bigr),
\qquad z^{(j)}(0)=Z_j,\quad 0\leq j\leq k.
\]
In particular, writing $w_j=\partial_t^j u$, we have
$w_j\in C([0,T];H_0^1(\Omega))$ for $0\leq j\leq k$ and
$w_{k+1}\in C([0,T];L^2(\Omega))$.

It remains to recover the additional spatial regularity. For
$0\leq j\leq k-1$, differentiate \r{full_eq} $j$ times and split the memory
integral at $s=t$. The endpoint terms cancel because the past and forward
jets match at $t=0$, and we obtain
\[
\Delta_{\sigma_\varepsilon}w_j
-\varepsilon\nabla\cdot\int_0^t\mathfrak m(s,\cdot)
\nabla w_j(t-s,\cdot)\,\d s
=w_{j+2}+f_-^{(j)}-\partial_t^j f.
\]
The assumptions on $u_-$ and \r{m-int-k} imply
$f_-^{(j)}\in C([0,T];H^{k-1-j}(\Omega))$; this follows directly from the
usual Sobolev product estimate and dominated convergence in $s$. For
$j=k-1$, the right-hand side of the last equation is continuous with values
in $L^2(\Omega)$, so Lemma~\ref{lemma_V} gives
$w_{k-1}\in C([0,T];H^2(\Omega)\cap H_0^1(\Omega))$. Here the solution
provided by the lemma agrees with the energy solution obtained above: the
same Volterra argument, applied with
$\mathcal H_1=H_0^1(\Omega)$ and $\mathcal H_2=H^{-1}(\Omega)$, gives
uniqueness in the energy class. We now proceed downward in $j$. At each
step, $w_{j+2}$ is already continuous with values in
$H^{k-1-j}(\Omega)$; applying Lemma~\ref{lemma_V} with Sobolev index
$k-1-j$ gives $w_j\in C([0,T];H^{k+1-j}(\Omega)\cap H_0^1(\Omega))$.
This proves the claimed regularity. Existence and uniqueness in the weak
class follow from Theorems~\ref{thm_wp1} and~\ref{thm_wp2}, and continuous
dependence follows from the semigroup and Volterra estimates used above.
\end{proof}

\begin{remark}
We emphasize that the a priori weak solution to \r{meq3} is smooth when the data are smooth and compatible; treating it merely as a weak solution does not establish this directly. 
\end{remark}

\begin{corollary}\label{cor_est}
Let $k\ge 1$. Define
\[
\|u\|_{Y_k(T)} := \sum_{j=0}^{k+1}
\|\partial_t^j u\|_{C([0,T];\,H^{k+1-j}(\Omega))}.
\]
Then, under the assumptions of Theorem~\ref{thm_well_posed}, for $0\le \varepsilon\le\varepsilon_0$, the solution satisfies
\[
\begin{aligned}
\|u\|_{Y_k(T)}
&\le
C_{k,T}
\bigg(
\|u_0\|_{H^{k+1}(\Omega)}
+
\|v_0\|_{H^k(\Omega)}
\\
&\qquad
+
\sum_{\ell=0}^{k}
\|\partial_t^\ell u_-\|_{C_{\mathrm b}((-\infty,0];\,H^{k+1-\ell}(\Omega))}
+
\sum_{\ell=0}^{k}
\|\partial_t^\ell f\|_{C([0,T];\,H^{k-\ell}(\Omega))}
\bigg).
\end{aligned}
\]
Here $C_{k,T}$ depends on $k,T,\varepsilon_0$, the elliptic constants, and the kernel bounds, but not on the particular data or on
$\varepsilon\in[0,\varepsilon_0]$.
\end{corollary}

\begin{proof}
Each step in the proof of Theorem~\ref{thm_well_posed} comes with a
corresponding estimate. The assumptions on $\mathfrak m$ first give uniform
bounds for the maps from $u_-$ to the history elements $\eta_j^0$ and to the
terms $f_-^{(j)}$. The recursion \r{eq:forward-jets} then bounds the initial
jets $U_j$, $0\leq j\leq k+1$, by the expression in parentheses above. The
contraction estimate for the semigroup gives the corresponding bounds in the
energy and graph spaces. Finally, the solution operator in
Lemma~\ref{lemma_V} is uniformly bounded for
$0\leq\varepsilon\leq\varepsilon_0$; applying that estimate successively,
starting with $\partial_t^{k-1}u$ and proceeding downward, gives the stated
$Y_k(T)$ estimate. The constants in all these steps are uniform on
$[0,\varepsilon_0]$, which proves the claim.
\end{proof}

\begin{remark}\label{rem_h1}
We could have formulated and proved the estimates above semiclassically, replacing the derivatives in the estimates and in the definitions of the Sobolev spaces by $h\partial_t$ and $h\partial_x$. In fact, this would have been more convenient in the next step, providing sharper constants, which we do not track here; however, it would have slightly complicated the present step.
\end{remark}

\subsection{Justification of the asymptotic expansion; proof of Theorem~\ref{thm_exp}} \label{sec_just}
We have constructed earlier, for every integer $N>1$, an asymptotic solution $u_N$ to \r{1c}, \r{IC1} in $[0,T]\times\Omega$ solving $P_hu_N=r_N$, where $r_N=O(h^N)$ in every $C_h^m$. Consequently, for every fixed integer $K\geq1$, the source norm on the right-hand side of the estimate in Corollary~\ref{cor_est}, with $f=r_N$, is $O(h^{N-K})$.

The prehistory $u_-$ vanishes in $B(0,R)$, and so do $u_0$ and $v_0$. We can take $u_N$ to agree with the free solution for $t<t_0$, and hence $r_N=0$ there. Consider the correction problem $P_hw_N=r_N$ with zero Cauchy data at $t=0$ and zero prehistory. Since $r_N$ vanishes in a neighborhood of $t=0$, all of its time derivatives at $t=0$ vanish. Thus the forward jets of the correction are all zero, and its regularity and compatibility conditions are satisfied to every order. Theorem~\ref{thm_well_posed} and Corollary~\ref{cor_est}, applied at order $K$, give
\begin{equation}   \label{uw}
\|w_N\|_{Y_K(T)}\leq C_{K,T}h^{N-K}.
\end{equation}
Now fix the desired classical regularity order $q\geq0$ and precision order $\ell>0$. Choose an integer $K\geq1$ such that $K+1>q+n/2$. Sobolev embedding in the spatial variables, applied to each time derivative in the definition of $Y_K(T)$, yields $\|w_N\|_{C^q([0,T]\times\Omega)}\leq C\|w_N\|_{Y_K(T)}$. Choosing an integer $N\geq K+\ell$, we conclude from \r{uw} that $\|w_N\|_{C^q([0,T]\times\Omega)}\leq Ch^\ell$. Therefore, $u:=u_N-w_N$ is an exact solution to \r{1b}, \r{IC1}, and
\[
\|u-u_N\|_{C^q([0,T]\times\Omega)}\leq Ch^\ell.
\]
The exact solutions obtained from different truncation orders coincide by uniqueness in the weak class. Since $K$ can be chosen arbitrarily large, this solution is smooth. Finally, choose $\Omega$ so large that the causal support of the data does not meet $\partial\Omega$ before time $T$. The finite propagation argument at the beginning of section~\ref{sec_wp} shows that the solution is independent of this choice of $\Omega$ in the region under consideration and therefore gives the asserted solution on $[0,T]\times\R^n$.

As a consequence, we get the following. 

\begin{proposition} \label{pr_u_N}
There exists a unique smooth solution to \r{1b}, \r{IC1} on $[0,T]\times\R^n$.
Given an integer $q\geq0$ and $\ell>0$, there exists an integer $N$ so that for the finite asymptotic expansion $u_N$ constructed above we have
\[
\|u-u_N\|_{C^q([0,T]\times\R^n)}\leq Ch^\ell.
\]
\end{proposition}

\begin{remark}
The $C^q$ norm above can be replaced by the $C_h^q$ norm, which is more natural for the solutions we consider. As we mentioned in Remark~\ref{rem_h1}, we could have worked in \sc\ spaces all the time. 
\end{remark}

\section*{Acknowledgements}
The authors thank Anna Mazzucato for helpful discussions on viscoelastic wave models.

\bibliographystyle{abbrv}
\bibliography{visco_references}

@Article{AcostaM_16,
  author  = {Sebastián Acosta and Carlos Montalto},
  title   = {Photoacoustic imaging taking into account thermodynamic attenuation},
  journal = {Inverse Problems},
  year    = {2016},
  volume  = {32},
  number  = {11},
  pages   = {115001},
  note    = {http://stacks.iop.org/0266-5611/32/i=11/a=115001},
}

@Book{Friedlander1998,
  author    = {Friedlander, Friedrich Gerard and Joshi, Mark Suresh},
  title     = {Introduction to the Theory of Distributions},
  publisher = {Cambridge University Press},
  year      = {1998},
}

@Book{Hormander1,
  Title                    = {The analysis of linear partial differential operators. {I}},
  Author                   = {H{\"o}rmander, Lars},
  Publisher                = {Springer-Verlag},
  Year                     = {1983},

  Address                  = {Berlin},
  Note                     = {Distribution theory and Fourier analysis},
  Volume                   = {256},

  ISBN                     = {3-540-12104-8},
  Mrclass                  = {35-02 (42B10 46Fxx)},
  Mrnumber                 = {MR717035 (85g:35002a)},
  Mrreviewer               = {L. Cattabriga},
  Pages                    = {ix+391}
}

@Article{Andrew13,
  Title                    = {Multi-wave imaging in attenuating media},
  Author                   = {Homan, Andrew},
  Journal                  = {Inverse Probl. Imaging},
  Year                     = {2013},
  Number                   = {4},
  Pages                    = {1235--1250},
  Volume                   = {7},

  Doi                      = {10.3934/ipi.2013.7.1235},
  Fjournal                 = {Inverse Problems and Imaging},
  ISSN                     = {1930-8337},
  Mrclass                  = {35R30 (35L15 92C55)},
  Mrnumber                 = {3180678},
  Mrreviewer               = {Markus Haltmeier},
  Url                      = {http://dx.doi.org/10.3934/ipi.2013.7.1235}
}

@Article{Novikov,
  Title                    = {An inversion formula for the attenuated {X}-ray transformation},
  Author                   = {Novikov, Roman G.},
  Journal                  = {Ark. Mat.},
  Year                     = {2002},
  Number                   = {1},
  Pages                    = {145--167},
  Volume                   = {40},

  Coden                    = {AKMTAJ},
  Fjournal                 = {Arkiv f\"or Matematik},
  ISSN                     = {0004-2080},
  Mrclass                  = {44A12 (35R30 65R10 82C70 92C55)},
  Mrnumber                 = {MR1948891 (2003k:44004)},
  Mrreviewer               = {Olga Klimenko}
}

@Article{SU-thermo_brain,
  Title                    = {Thermoacoustic tomography arising in brain imaging},
  Author                   = {Stefanov, Plamen and Uhlmann, Gunther},
  Journal                  = {Inverse Problems},
  Year                     = {2011},
  Number                   = {4},
  Pages                    = {045004, 26},
  Volume                   = {27},

  Coden                    = {INPEEY},
  Doi                      = {10.1088/0266-5611/27/4/045004},
  Fjournal                 = {Inverse Problems. An International Journal on the Theory and Practice of Inverse Problems, Inverse Methods and Computerized Inversion of Data},
  ISSN                     = {0266-5611},
  Mrclass                  = {76Q05 (35L05 35R30 65M32 92C55)},
  Mrnumber                 = {2781028},
  Url                      = {http://dx.doi.org/10.1088/0266-5611/27/4/045004}
}

@Article{SU-thermo,
  Title                    = {Thermoacoustic tomography with variable sound speed},
  Author                   = {Stefanov, Plamen and Uhlmann, Gunther},
  Journal                  = {Inverse Problems},
  Year                     = {2009},
  Number                   = {7},
  Pages                    = {075011, 16},
  Volume                   = {25},

  Coden                    = {INPEEY},
  Doi                      = {10.1088/0266-5611/25/7/075011},
  Fjournal                 = {Inverse Problems. An International Journal on the Theory and Practice of Inverse Problems, Inverse Methods and Computerized Inversion of Data},
  ISSN                     = {0266-5611},
  Mrclass                  = {35R30 (76Q05)},
  Mrnumber                 = {MR2519863},
  Url                      = {http://dx.doi.org/10.1088/0266-5611/25/7/075011}
}

@Book{Reed-Simon2,
  author     = {Reed, Michael and Simon, Barry},
  title      = {Methods of modern mathematical physics. {II}. {F}ourier analysis, self-adjointness},
  year       = {1975},
  publisher  = {Academic Press [Harcourt Brace Jovanovich, Publishers], New York-London},
  pages      = {xv+361},
  mrclass    = {47-02 (81.47)},
  mrnumber   = {0493420},
  mrreviewer = {P. R. Chernoff},
}

@Article{Chep_Pata,
  author     = {Chepyzhov, V. V. and Pata, V.},
  title      = {Some remarks on stability of semigroups arising from linear viscoelasticity},
  doi        = {10.3233/asy-2006-731},
  issn       = {0921-7134,1875-8576},
  number     = {3-4},
  pages      = {251--273},
  url        = {https://doi-org.ezproxy.lib.purdue.edu/10.3233/asy-2006-731},
  volume     = {46},
  fjournal   = {Asymptotic Analysis},
  journal    = {Asymptot. Anal.},
  mrclass    = {47D06 (35B35 35Q72 45K05 74D05 74H55)},
  mrnumber   = {2215885},
  mrreviewer = {Rosella\ Sampalmieri},
  year       = {2006},
}

@Article{Giorgi_P_01,
  author     = {Giorgi, Claudio and Pata, Vittorino},
  title      = {Stability of abstract linear thermoelastic systems with memory},
  doi        = {10.1142/S0218202501001021},
  issn       = {0218-2025,1793-6314},
  number     = {4},
  pages      = {627--644},
  url        = {https://doi-org.ezproxy.lib.purdue.edu/10.1142/S0218202501001021},
  volume     = {11},
  fjournal   = {Mathematical Models and Methods in Applied Sciences},
  journal    = {Math. Models Methods Appl. Sci.},
  mrclass    = {74D05 (34G10 35R10 74F05 74H40)},
  mrnumber   = {1832996},
  mrreviewer = {Gustavo\ Perla\ Menzala},
  year       = {2001},
}

@Article{HrusaNR_88,
  author   = {Hrusa, W. J. and Nohel, J. A. and Renardy, M.},
  title    = {Initial Value Problems in Viscoelasticity},
  doi      = {10.1115/1.3151871},
  issn     = {0003-6900},
  number   = {10},
  pages    = {371--378},
  url      = {https://doi.org/10.1115/1.3151871},
  volume   = {41},
  journal  = {Applied Mechanics Reviews},
  month    = {10},
  year     = {1988},
}

@Article{Hanyga2002PropagationPulses,
  author  = {Hanyga, Andrzej},
  title   = {Propagation of Pulses in Viscoelastic Media},
  doi     = {10.1007/s00024-002-8707-x},
  pages   = {1749--1769},
  volume  = {159},
  journal = {Pure and Applied Geophysics},
  year    = {2002},
}

@Article{HrusaR_85,
  author     = {Hrusa, W. J. and Renardy, M.},
  title      = {On wave propagation in linear viscoelasticity},
  doi        = {10.1090/qam/793532},
  issn       = {0033-569X,1552-4485},
  number     = {2},
  pages      = {237--254},
  url        = {https://doi-org.ezproxy.lib.purdue.edu/10.1090/qam/793532},
  volume     = {43},
  fjournal   = {Quarterly of Applied Mathematics},
  journal    = {Quart. Appl. Math.},
  mrclass    = {45K05 (73F99)},
  mrnumber   = {793532},
  mrreviewer = {Robert\ L.\ Wheeler},
  year       = {1985},
}

@Article{HanygaSeredynska2002SingularMemory,
  author  = {Hanyga, Andrzej and Seredy{\'n}ska, Ma{\l}gorzata},
  title   = {Asymptotic and exact fundamental solutions in hereditary media with singular memory kernels},
  doi     = {10.1090/S0033-569X-02-01080-5},
  number  = {2},
  pages   = {213--244},
  volume  = {60},
  journal = {Quarterly of Applied Mathematics},
  year    = {2002},
}

@Article{HanygaSeredynska1999MemoryKernelSingularityI,
  author  = {Hanyga, Andrzej and Seredy{\'n}ska, Ma{\l}gorzata},
  title   = {Some effects of the memory kernel singularity on wave propagation and inversion in poroelastic media, {I}: Forward modeling},
  doi     = {10.1046/j.1365-246X.1999.00775.x},
  number  = {2},
  pages   = {319--335},
  volume  = {137},
  journal = {Geophysical Journal International},
  year    = {1999},
}

@Article{Dafermos_70,
  author     = {Dafermos, Constantine M.},
  title      = {Asymptotic stability in viscoelasticity},
  doi        = {10.1007/BF00251609},
  issn       = {0003-9527},
  pages      = {297--308},
  url        = {https://doi-org.ezproxy.lib.purdue.edu/10.1007/BF00251609},
  volume     = {37},
  fjournal   = {Archive for Rational Mechanics and Analysis},
  journal    = {Arch. Rational Mech. Anal.},
  mrclass    = {73.45},
  mrnumber   = {281400},
  mrreviewer = {H.\ Amann},
  year       = {1970},
}

@InProceedings{Dafermos_76,
  author    = {Dafermos, C. M.},
  booktitle = {Applications of Methods of Functional Analysis to Problems in Mechanics},
  title     = {Contraction semigroups and trend to equilibrium in continuum mechanics},
  editor    = {Germain, Paul and Nayroles, Bernard},
  isbn      = {978-3-540-38165-5},
  pages     = {295--306},
  publisher = {Springer Berlin Heidelberg},
  address   = {Berlin, Heidelberg},
  year      = {1976},
}

@Book{Pruss1993,
  author    = {Pr{\"u}ss, Jan},
  title     = {Evolutionary Integral Equations and Applications},
  doi       = {10.1007/978-3-0348-0499-8},
  publisher = {Birkh{\"a}user Verlag},
  series    = {Monographs in Mathematics},
  volume    = {87},
  address   = {Basel},
  year      = {1993},
}

@Article{ChintadaRauGoksel2022,
  author  = {Chintada, Bhaskara Rao and Rau, Richard and Goksel, Orcun},
  title   = {Spectral Ultrasound Imaging of Speed-of-Sound and Attenuation Using an Acoustic Mirror},
  doi     = {10.3389/fphy.2022.860725},
  issn    = {2296-424X},
  url     = {https://www.frontiersin.org/journals/physics/articles/10.3389/fphy.2022.860725},
  volume  = {10},
  journal = {Frontiers in Physics},
  year    = {2022},
}

@Book{Bushberg2012,
  author    = {Bushberg, Jerrold T. and Seibert, J. Anthony and Leidholdt, Jr., Edwin M. and Boone, John M.},
  title     = {The Essential Physics of Medical Imaging},
  edition   = {3},
  isbn      = {9780781780575},
  publisher = {Lippincott Williams \& Wilkins},
  address   = {Philadelphia, PA},
  year      = {2012},
}

@Misc{AIUM2019,
  author       = {{American Institute of Ultrasound in Medicine}},
  title        = {Statement Regarding Mathematical Tissue Models},
  howpublished = {\url{https://www.aium.org/resources/official-statements/view/statement-regarding-mathematical-tissue-models}},
  note         = {Approved March 19, 2007; reapproved April 1, 2012, April 7, 2019, and July 30, 2024. Accessed June 13, 2026},
  year         = {2019},
}

@Article{Kjartansson1979,
  author  = {Kjartansson, E.},
  title   = {Constant Q-wave propagation and attenuation},
  doi     = {10.1029/JB084iB09p04737},
  number  = {B9},
  pages   = {4737--4748},
  volume  = {84},
  journal = {Journal of Geophysical Research},
  year    = {1979},
}

@Article{LekicMatasPanningRomanowicz2009,
  author  = {Leki{\'c}, Ved and Matas, Jan and Panning, Mark P. and Romanowicz, Barbara},
  title   = {Measurement and implications of frequency dependence of attenuation},
  doi     = {10.1016/j.epsl.2009.03.030},
  number  = {1--4},
  pages   = {285--293},
  volume  = {282},
  journal = {Earth and Planetary Science Letters},
  year    = {2009},
}

@Book{Mainardi2010,
  author    = {Francesco Mainardi},
  title     = {Fractional Calculus and Waves in Linear Viscoelasticity: An Introduction to Mathematical Models},
  isbn      = {978-1-84816-329-4},
  publisher = {Imperial College Press},
  address   = {London},
  year      = {2010},
}

@Article{Covi_26,
  author        = {Giovanni Covi and Maarten de Hoop and Mikko Salo},
  title         = {Propagation of singularities and inverse problems for the viscoacoustic wave equation},
  eprint        = {2603.24497},
  url           = {https://arxiv.org/abs/2603.24497},
  archiveprefix = {arXiv},
  journal       = {arXiv:2603.24497},
  primaryclass  = {math.AP},
  year          = {2026},
}

@Article{acosta2025,
  author        = {Sebastian Acosta and Benjamin Palacios},
  title         = {Inverse photoacoustic tomography problem in media with fractional attenuation},
  eprint        = {2510.03408},
  url           = {https://arxiv.org/abs/2510.03408},
  archiveprefix = {arXiv},
  journal       = {arXiv:2510.03408},
  primaryclass  = {math.AP},
  year          = {2025},
}

@Article{Hanyga2001SingularMemory,
  author  = {Hanyga, Andrzej},
  title   = {Wave propagation in media with singular memory},
  doi     = {10.1016/S0895-7177(01)00137-6},
  number  = {12--13},
  pages   = {1399--1421},
  volume  = {34},
  journal = {Mathematical and Computer Modelling},
  year    = {2001},
}

@Article{Szabo1994PowerLaw,
  author  = {Szabo, Thomas L.},
  title   = {Time domain wave equations for lossy media obeying a frequency power law},
  doi     = {10.1121/1.410434},
  number  = {1},
  pages   = {491--500},
  volume  = {96},
  journal = {The Journal of the Acoustical Society of America},
  year    = {1994},
}

@Article{HolmNasholm2014,
  author  = {Holm, Sverre and N{\"a}sholm, Sven Peter},
  title   = {Comparison of fractional wave equations for power law attenuation in ultrasound and elastography},
  doi     = {10.1016/j.ultrasmedbio.2013.09.033},
  number  = {4},
  pages   = {695--703},
  volume  = {40},
  journal = {Ultrasound in Medicine \& Biology},
  year    = {2014},
}

@Article{BrouckeOparnica2022,
  author  = {Broucke, Frederik and Oparnica, Ljubica},
  title   = {Micro-local and qualitative analysis of the fractional {Z}ener wave equation},
  doi     = {10.1016/j.jde.2022.03.006},
  pages   = {217--257},
  volume  = {321},
  journal = {Journal of Differential Equations},
  year    = {2022},
}

@Article{Sinestrari1999,
  author  = {Sinestrari, Eugenio},
  title   = {Wave equation with memory},
  doi     = {10.3934/dcds.1999.5.881},
  number  = {4},
  pages   = {881--896},
  volume  = {5},
  journal = {Discrete and Continuous Dynamical Systems},
  year    = {1999},
}

@Article{KaltenbacherEtAl2022,
  author  = {Kaltenbacher, Barbara and Khristenko, Ustim and Nikoli{\'c}, Vanja and Rajendran, Mabel Lizzy and Wohlmuth, Barbara},
  title   = {Determining kernels in linear viscoelasticity},
  doi     = {10.1016/j.jcp.2022.111331},
  pages   = {111331},
  volume  = {464},
  journal = {Journal of Computational Physics},
  year    = {2022},
}

@Article{BukhgeimDyatlovUhlmann2007,
  author  = {Bukhgeim, A. L. and Dyatlov, G. V. and Uhlmann, Gunther},
  title   = {Unique continuation for hyperbolic equations with memory},
  doi     = {10.1515/JIIP.2007.032},
  number  = {6},
  pages   = {587--598},
  volume  = {15},
  journal = {Journal of Inverse and Ill-Posed Problems},
  year    = {2007},
}

@InCollection{BukhgeimDyatlovUhlmann2002,
  author    = {Bukhgeim, A. L. and Dyatlov, G. V. and Uhlmann, Gunther},
  title     = {Reconstruction of the memory from partial boundary measurements},
  booktitle = {Mathematical Results in Quantum Mechanics},
  series    = {Contemporary Mathematics},
  volume    = {307},
  pages     = {39--46},
  publisher = {American Mathematical Society},
  address   = {Providence, RI},
  year      = {2002},
}

@Article{ColemanNoll1961,
  author  = {Coleman, Bernard D. and Noll, Walter},
  title   = {Foundations of Linear Viscoelasticity},
  doi     = {10.1103/RevModPhys.33.239},
  number  = {2},
  pages   = {239--249},
  volume  = {33},
  journal = {Reviews of Modern Physics},
  year    = {1961},
}

@Book{Lakes2009,
  author    = {Lakes, Roderic S.},
  title     = {Viscoelastic Materials},
  doi       = {10.1017/CBO9780511626722},
  isbn      = {978-0-521-88568-3},
  publisher = {Cambridge University Press},
  address   = {Cambridge},
  year      = {2009},
}

@Article{BagleyTorvik1983,
  author  = {Bagley, Ronald L. and Torvik, Peter J.},
  title   = {A Theoretical Basis for the Application of Fractional Calculus to Viscoelasticity},
  doi     = {10.1122/1.549724},
  number  = {3},
  pages   = {201--210},
  volume  = {27},
  journal = {Journal of Rheology},
  year    = {1983},
}

@Article{Romanov2014Coefficients,
  author  = {Romanov, V. G.},
  title   = {On the Determination of the Coefficients in the Viscoelasticity Equations},
  doi     = {10.1134/S0037446614030124},
  number  = {3},
  pages   = {503--510},
  volume  = {55},
  journal = {Siberian Mathematical Journal},
  year    = {2014},
}

\end{document}